\documentclass[10pt,%
	final,
	notitlepage,
	narroweqnarray,
	inline,
	twoside,
    letterpaper,
	]{IEEEtran}

\usepackage{algorithmic}
\usepackage{algorithm}
\usepackage{amsmath}
\usepackage{amsfonts}
\usepackage[pdftex]{graphicx} 
\usepackage{epstopdf}
\usepackage{epsfig} 
\usepackage{times} 
\usepackage{amssymb}  
\usepackage{amsthm}
\usepackage{float, color}
\usepackage{cite, mathabx}
\usepackage[caption=true,font=normalsize,labelfont=sf,textfont=sf]{subfig}
\usepackage{fancyvrb} 
\usepackage{multirow}

\newtheorem{lem}{Lemma}

\newtheorem{thm}{Theorem}

\begin{document}
\bstctlcite{IEEEexample:BSTcontrol}

\title{Multiscale Analysis for Higher-order Tensors}


    
\author{Alp Ozdemir, Ali Zare, Mark A. Iwen, and Selin Aviyente \thanks{\footnotesize 
A. Ozdemir (ozdemira@egr.msu.edu) and S. Aviyente (aviyente@egr.msu.edu) are with the Electrical and Computer Engineering, Michigan State University, East Lansing, MI, 48824, USA.   

Ali Zare (zareali@msu.edu) is with the Department of Computational Mathematics, Science, and Engineering (CMSE), Michigan State University, East Lansing, MI, 48824, USA.

Mark A. Iwen (markiwen@math.msu.edu) is with the Department of Mathematics, and the Department of Computational Mathematics, Science, and Engineering (CMSE), Michigan State University, East Lansing, MI, 48824, USA.  

This work was in part supported by NSF DMS-1416752 and NSF CCF-1615489.
}}

\onecolumn
\pagenumbering{gobble}
\maketitle               
\parskip=0 pt
\begin{abstract}
The widespread use of multisensor technology and the emergence of big datasets have created the need to develop tools to reduce, approximate, and classify large and multimodal data such as higher-order tensors. While early approaches focused on matrix and vector based methods to represent these higher-order data, more recently it has been shown that tensor decomposition methods are better equipped to capture couplings across their different modes. For these reasons, tensor decomposition methods have found applications in many different signal processing problems including dimensionality reduction, signal separation, linear regression, feature extraction, and classification. However, most of the existing tensor decomposition methods are based on the principle of finding a low-rank approximation in a linear subspace structure, where the definition of the rank may change depending on the particular decomposition. Since many datasets are not necessarily low-rank in a linear subspace, this often results in high approximation errors or low compression rates. In this paper, we introduce a new adaptive, multi-scale tensor decomposition method for higher order data inspired by hybrid linear modeling and subspace clustering techniques. In particular, we develop a multi-scale higher-order singular value decomposition (MS-HoSVD) approach where a given tensor is first permuted and then partitioned into several sub-tensors each of which can be represented as a low-rank tensor with increased representational efficiency. The proposed approach is evaluated for dimensionality reduction and classification for several different real-life tensor signals with promising results.
\end{abstract}


\section{Introduction}

Data in the form of multidimensional arrays, also referred to as tensors, arise in a variety of applications including chemometrics, hyperspectral imaging, high resolution videos, neuroimaging, biometrics and social network analysis \cite{letexier2008nonorthogonal, kim2009canonical, miwakeichi2004decomposing}. These applications produce massive amounts of data collected in various forms with multiple aspects and high dimensionality. Tensors, which are multi-dimensional generalizations of matrices, provide a useful representation for such data. A crucial step in many applications involving higher-orders tensors is multiway reduction of the data to ensure that the reduced representation of the tensor retains certain characteristics. Early multiway data analysis approaches reformatted the tensor data as a matrix and resorted to methods developed for classical two-way analysis. However, one cannot discover hidden components within multiway data using conventional matrix decomposition methods as matrix based representations cannot capture multiway couplings focusing on standard pairwise interactions. To this end, many different types of tensor decomposition methods have been proposed in  literature \cite{cichocki2015tensor,oseledets2011tensor,de2008decompositions2,Merhi_Face_2012,cheng2015probabilistic,fu2015joint,zareTensorPCA}.

In contrast to the matrix case where data reduction is often accomplished via low-rank representations such as singular value decomposition (SVD), the notion of rank for higher order tensors is not uniquely defined. The CANDECOMP/PARAFAC (CP) and Tucker decompositions are two of the most widely used tensor decomposition methods for data reduction \cite{kolda2009tensor,de2000multilinear}. For CP, the goal is to approximate the given tensor as a weighted sum of rank-1 tensors, where a rank-1 tensor refers to the outer product of $n$ vectors with $n$ being equal to the order of the tensor.  The Tucker model allows for interactions between the factors from different modes resulting in a typically dense, but small, core tensor. This model also introduces the notion of Tucker rank or n-rank, which refers to the n-tuple of ranks corresponding to the tensor's unfoldings along each of its modes. Therefore, low rank approximation with the Tucker model can be obtained by projections onto low-rank factor matrices. Unlike the CP decomposition, the Tucker decomposition is in general non-unique. To help obtain meaningful and unique representations by the Tucker decomposition, orthogonality, sparsity, and non-negativity constraints are often imposed on the factors yielding, e.g., the Non-Negative Tensor Factorization (NTF) and the Sparse Non-Negative Tucker Decomposition \cite{shashua2005non,cichocki2007non,cichocki2007novel}. The Tucker decomposition with orthogonality constraints on the factors is known as the Higher-Order Singular Value Decomposition (HoSVD), or Multilinear SVD \cite{de2000multilinear}.  The HoSVD can be computed by simply flattening the tensor in each mode and calculating the n-mode singular vectors corresponding to that mode.  





   With the emergence of multidimensional big data, classical tensor representation and decomposition methods have become inadequate since the size of these tensors exceeds available working memory and the processing time is very long. In order to address the problem of large-scale tensor decomposition, several block-wise tensor decomposition methods have been proposed \cite{de2008decompositions2}. The basic idea is to partition a big data tensor into smaller blocks and perform tensor related operations block-wise using a suitable tensor format. Preliminary approaches relied on a hierarchical tree structure and reduced the storage of d-dimensional arrays to the storage of auxiliary three-dimensional ones such as the tensor-train decomposition (T-Train), also known as the matrix product state (MPS) decomposition, \cite{oseledets2011tensor} and the Hierarchical Tucker Decomposition (H-Tucker) \cite{grasedyck2010hierarchical}. In particular, in the area of large volumetric data visualization, tensor based multiresolution hierarchical methods such as TAMRESH have attracted attention \cite{suter2013tamresh}. However, all of these methods are interested in fitting a low-rank model to data which lies near a \textit{linear} subspace, thus being limited to learning linear structure. 
   
 Similar to the research efforts in tensor reduction, low-dimensional subspace and manifold learning methods have also been extended for higher order data clustering and classification applications. In early work in the area, Vasilescu and Terzopoulos \cite{vasilescu2002multilinear} extended the eigenface concept to the tensorface by using higher order SVD and taking different modes such as expression, illumination and pose into account. Similarly,  2D-PCA for matrices has been used for feature extraction from face images without converting the images into vectors \cite{yang2004two}.   He et al. \cite{he2005tensor} extended locality preserving projections \cite{niyogi2004locality} to second order tensors for face recognition. Dai and Yeung \cite{dai2006tensor} presented generalized tensor embedding methods such as the extensions of local discriminant embedding methods \cite{chen2005local}, neighborhood preserving embedding methods \cite{he2005neighborhood}, and locality preserving projection methods \cite{niyogi2004locality} to tensors.  Li et al. \cite{li2008discriminant} proposed a supervised manifold learning method for vector type data which preserves local structures in each class of samples, and then extended the algorithm to  tensors to provide improved performance for face and gait recognition. Similar to vector-type manifold learning algorithms, the aim of these methods is to find an optimal \textit{linear} transformation for the tensor-type training data samples without vectorizing them and mapping these samples to a low dimensional subspace while preserving the neighborhood information. 
 

In this paper, we propose a novel multi-scale analysis technique to efficiently approximate tensor type data using locally linear low-rank approximations. The proposed method consists of two major steps: 1) Constructing a tree structure by partitioning the tensor into a collection of permuted subtensors, followed by 2) Constructing multiscale dictionaries by applying HoSVD to each subtensor.  The contributions of the proposed framework and the novelty in the proposed approach with respect to previously published work in \cite{ozdemir2016multiscale,ozdemir2017multi} are manifold.  They include: 1) The introduction of a more flexible multi-scale tensor decomposition method which allows the user to approximate a given tensor within given memory and processing power constraints; 2) the introduction of theoretical error bounds for the proposed decomposition;  3) the introduction of adaptive pruning to achieve a better trade-off between compression rate and reconstruction error for the developed factorizations; 4) the extensive evaluation of the method for both data reduction and classification applications; and 5) a detailed comparison of the proposed method to state-of-the-art tensor decomposition methods including the HoSVD, T-Train, and H-Tucker decompositions. 



The remainder of the paper is organized as follows. In Section~\ref{sec:background}, basic notation and tensor operations are reviewed.  The proposed multiscale tensor decomposition method along with theoretical error bounds are then introduced in Section~\ref{multiScaleMethod}.
Sections~\ref{sec:DataCompress} and \ref{sec:Classification} illustrate the  results of applying the proposed framework to data reduction and classification problems, respectively.

\section{Background}
\label{sec:background}

\subsection{Tensor Notation and Algebra}
A multidimensional array with $N$ modes $\mathcal{X}\in \mathbb{R}^{I_1\times I_2\times ... \times I_N}$ is called a tensor, where $x_{i_1,i_2,..i_N}$ denotes the $({i_1,i_2,..i_N})^{th}$ element of the tensor $\mathcal{X}$. The vectors in $\mathbb{R}^{I_n}$ obtained by fixing all of the indices of such a tensor $\mathcal{X}$ except for the one that corresponds to its $n$th mode are called its {\it mode-$n$ fibers}. Let $[N] := \{1, \dots, N \}$ for all $N \in \mathbb{N}$.  Basic tensor operations are reviewed below (see, e.g., \cite{kolda2009tensor}, \cite{de2008tensor}, \cite{vannieuwenhoven2012new}).

\noindent {\bf Tensor addition and multiplication by a scalar:} Two tensors $\mathcal{X}, \mathcal{Y} \in \mathbb{R}^{I_1\times I_2\times ... \times I_N}$ can be added using component-wise tensor addition.  The resulting tensor $\mathcal{X} + \mathcal{Y} \in \mathbb{R}^{I_1\times I_2\times ... \times I_N}$ has its entries given by $\left( \mathcal{X} + \mathcal{Y} \right)_{i_1,i_2,..i_N} = x_{i_1,i_2,..i_N} + y_{i_1,i_2,..i_N}$.  Similarly, given a scalar $\alpha \in \mathbb{R}$ and a tensor $\mathcal{X} \in \mathbb{R}^{I_1\times I_2\times ... \times I_N}$ the rescaled tensor $\alpha \mathcal{X}  \in \mathbb{R}^{I_1\times I_2\times ... \times I_N}$ has its entries given by $\left( \alpha \mathcal{X} \right)_{i_1,i_2,..i_N} = \alpha~x_{i_1,i_2,..i_N} $.

\noindent {\bf Mode-$n$ products:} The mode-$n$ product of a tensor $\mathcal{X}\in \mathbb{R}^{I_1\times ... I_n\times ...\times I_N}$ and a matrix ${\bf{U}} \in \mathbb{R}^{J\times I_n} $ is denoted as $\mathcal{Y}=\mathcal{X} \times_{n} {\bf{U}}$, $(\mathcal{Y})_{i_{1},i_{2},\ldots,i_{n-1},j,i_{n+1},\ldots,i_{N}}=\sum_{i_{n}=1}^{I_{n}}x_{i_{1},\ldots,i_{n},\ldots,i_{N}}u_{j,i_{n}}$.  It is of size $I_1\times ...\times I_{n-1} \times J \times I_{n+1}\times ...\times I_N$.  The following facts about mode-$n$ products are useful (see, e.g., \cite{kolda2009tensor},\cite{vannieuwenhoven2012new}).

\begin{lem}
Let $\mathcal{X}, \mathcal{Y} \in \mathbb{R}^{I_1\times I_2\times ... \times I_N}$, $\alpha, \beta \in \mathbb{R}$, and ${\bf U}^{(n)}, {\bf V}^{(n)} \in \mathbb{R}^{J_n \times I_n}$ for all $n \in [N]$.  The following are true:
\begin{enumerate}
\item[($a$)] $\left( \alpha \mathcal{X} + \beta \mathcal{Y} \right) \times_n {\bf U}^{(n)} = \alpha \left( \mathcal{X} \times_n {\bf U}^{(n)} \right) + \beta \left( \mathcal{Y} \times_n {\bf U}^{(n)} \right)$.

\item[($b$)] $\mathcal{X} \times_n  \left( \alpha {\bf U}^{(n)} + \beta {\bf V}^{(n)} \right) = \alpha \left( \mathcal{X} \times_n {\bf U}^{(n)} \right) + \beta \left( \mathcal{X} \times_n {\bf V}^{(n)} \right)$. 

\item[($c$)] If $n \neq m$ then $\mathcal{X} \times_n {\bf U}^{(n)} \times_m {\bf V}^{(m)} = \left( \mathcal{X} \times_n {\bf U}^{(n)} \right) \times_m {\bf V}^{(m)} = \left( \mathcal{X} \times_m {\bf V}^{(m)} \right) \times_n {\bf U}^{(n)} = \mathcal{X} \times_m {\bf V}^{(m)} \times_n {\bf U}^{(n)}$ .

\item[($d$)] If ${\bf W} \in \mathbb{C}^{P \times J_n}$ then $\mathcal{X} \times_n {\bf U}^{(n)} \times_n {\bf W} = \left( \mathcal{X} \times_n {\bf U}^{(n)} \right) \times_n {\bf W} = \mathcal{X} \times_n \left( {\bf WU}^{(n)} \right) = \mathcal{X} \times_n {\bf WU}^{(n)} $.
\end{enumerate}

\label{lem:modeProdProps}
\end{lem}  

\noindent {\bf Tensor matricization:} The process of reordering the elements of the tensor into a matrix is known as matricization or unfolding. The mode-n matricization of a tensor $\mathcal{Y} \in \mathbb{R}^{I_1\times I_2\times ... \times I_N}$ is denoted as ${\bf Y}_{(n)} \in \mathbb{R}^{I_n \times \prod_{m \neq n} I_m}$ and is obtained by arranging $\mathcal{Y}$'s mode-n fibers to be the columns of the resulting matrix. Unfolding the tensor $\mathcal{Y}= \mathcal{X}\times_1 {\bf U}^{(1)}\times_2 {\bf U}^{(2)}...\times_N {\bf U}^{(N)} =: {\displaystyle \mathcal{X} \bigtimes_{n=1}^N {\bf U}^{(n)}}$ along mode-$n$ is equivalent to 
\begin{equation}
{\bf Y}_{(n)} = {\bf U}^{(n)} {\bf X}_{(n)}({\bf U}^{(N)} \otimes... {\bf U}^{(n+1)}\otimes {\bf U}^{(n-1)}...\otimes{\bf U}^{(1)} )^\top,
\label{equ:KronModenFlat}
\end{equation}
where $\otimes$ is the matrix Kronecker product.  In particular, \eqref{equ:KronModenFlat} implies that the matricization $\left( \mathcal{X} \times_n {\bf U}^{(n)} \right)_{(n)} = {\bf U}^{(n)} {\bf X}_{(n)}$.\footnote{Simply set ${\bf U}^{(m)} = {\bf I}$ (the identity) for all $m \neq n$ in \eqref{equ:KronModenFlat}.  This fact also easily follows directly from the definition of the mode-$n$ product.}

It is worth noting that trivial inner product preserving isomorphisms exist between a tensor space $\mathbb{R}^{I_1\times I_2\times ... \times I_N}$ and any of its matricized versions (i.e., mode-$n$ matricization can be viewed as an isomorphism between the original tensor vector space $\mathbb{R}^{I_1\times I_2\times ... \times I_N}$ and its mode-$n$ matricized target vector space $\displaystyle \mathbb{R}^{I_n \times \prod_{m \neq n} I_m}$).  In particular, the process of matricizing tensors is linear.  If, for example, $\mathcal{X}, \mathcal{Y} \in \mathbb{R}^{I_1\times I_2\times ... \times I_N}$ then one can see that the mode-$n$ matricization of $\mathcal{X} + \mathcal{Y} \in \mathbb{R}^{I_1\times I_2\times ... \times I_N}$ is $\left( \mathcal{X} + \mathcal{Y} \right)_{(n)} = {\bf X}_{(n)} + {\bf Y}_{(n)}$ for all modes $n \in [N]$.  

\noindent {\bf Tensor Rank:}  Unlike matrices, which have a unique definition of rank, there are multiple rank definitions for tensors including \textit{tensor rank} and \textit{tensor n-rank}.
The \textit{rank} of a tensor $\mathcal{X}\in \mathbb{R}^{I_1\times ... I_n\times ...\times I_N}$ is the smallest number of rank-one tensors that form $\mathcal{X}$ as their sum. The \textit{n-rank} of $\mathcal{X}$ is the collection of ranks of unfoldings ${\bf X}_{(n)}$ and is denoted as: 
\begin{equation}
n\mbox{-rank}(\mathcal{X}) = \left( \mbox{rank}({\bf X}_{(1)}),\; \mbox{rank}( {\bf X}_{(2)}),...,\; \mbox{rank}( {\bf X}_{(N)}) \right).
\end{equation}

\noindent {\bf Tensor inner product:} The inner product of two same sized tensors $\mathcal{X}, \mathcal{Y}\in \mathbb{R}^{I_1\times I_2\times ... \times I_N}$ is the sum of the products of their elements.
\begin{equation}
\left\langle \mathcal{X},\mathcal{Y} \right\rangle = {\sum_{i_1=1}^{I_1} \sum_{i_2=1}^{I_2} ...\sum_{i_N=1}^{I_N} x_{i_1,i_2,...,i_N}y_{i_1,i_2,...,i_N}}.
\end{equation}  
It is not too difficult to see that matricization preserves Hilbert-Schmidt/Frobenius matrix inner products, i.e., $\left\langle \mathcal{X},\mathcal{Y} \right\rangle = \left\langle {\bf X}_{(n)}, {\bf Y}_{(n)} \right\rangle_{\rm F} = {\rm Trace}\left( {\bf X}_{(n)}^\top {\bf Y}_{(n)} \right)$ holds for all $n \in [N]$.  If $\left\langle \mathcal{X},\mathcal{Y} \right\rangle = 0$, $\mathcal{X}$ and $\mathcal{Y}$ are {\it orthogonal}.  

\noindent {\bf Tensor norm:} Norm of a tensor $\mathcal{X}\in \mathbb{R}^{I_1\times I_2\times ... \times I_N}$ is the square root of the sum of the squares of all its elements.
\begin{equation}
\parallel \mathcal{X} \parallel = \sqrt{\left\langle \mathcal{X},\mathcal{X} \right\rangle}= \sqrt{\sum_{i_1=1}^{I_1} \sum_{i_2=1}^{I_2} ...\sum_{i_N=1}^{I_N} x_{i_1,i_2,...,i_N}^2}.
\end{equation}  
The fact that matricization preserves Frobenius matrix inner products also means that it preserves Frobenius matrix norms.  As a result we have that $\left\| \mathcal{X} \right\| = \left\| {\bf X}_{(n)} \right\|_{\rm F}$ holds for all $n \in [N]$.  If $\mathcal{X}$ and $\mathcal{Y}$ are orthogonal and also have unit norm (i.e., have $\| \mathcal{X} \| = \| \mathcal{Y} \| = 1$) we will say that they are an {\it orthonormal} pair.

\subsection{Some Useful Facts Concerning Mode-$n$ Products and Orthogonality}

Let ${\bf I} \in \mathbb{R}^{I_n \times I_n}$ be the identity matrix.  Given a (low rank) orthogonal projection matrix ${\bf P} \in \mathbb{R}^{I_n \times I_n}$ one can decompose any given tensor $\mathcal{X} \in \mathbb{R}^{I_1\times I_2\times ... \times I_N}$ into two orthogonal tensors using Lemma~\ref{lem:modeProdProps} (b)
$$\mathcal{X} = \mathcal{X} \times_n {\bf I} = \mathcal{X} \times_n (({\bf I} - {\bf P}) + {\bf P}) = \mathcal{X} \times_n ({\bf I} - {\bf P}) + \mathcal{X} \times_n {\bf P}.$$
To check that the last two summands are orthogonal one can use \eqref{equ:KronModenFlat} to compute that
$$\left\langle \mathcal{X} \times_n ({\bf I} - {\bf P}),\mathcal{X} \times_n {\bf P} \right\rangle = \left\langle ({\bf I} - {\bf P}) {\bf X}_{(n)},  {\bf P} {\bf X}_{(n)} \right\rangle_{\rm F} = {\rm Trace}\left( {\bf X}_{(n)}^\top ({\bf I} - {\bf P}) {\bf P} {\bf X}_{(n)} \right) = 0.$$
As a result one can also verify that the Pythagorean theorem holds, i.e., that $\| \mathcal{X} \|^2 = \| \mathcal{X} \times_n {\bf P} \|^2 + \| \mathcal{X} \times_n ({\bf I} - {\bf P}) \|^2$.

If we now regard $\mathcal{X} \times_n {\bf P}$ as a low rank approximation to $\mathcal{X}$ then we can see that its approximation error
$$\mathcal{X} - \mathcal{X} \times_n {\bf P} = \mathcal{X} \times_n ({\bf I} - {\bf P})$$
is orthogonal to the low rank approximation $\mathcal{X} \times_n {\bf P}$, as one would expect.  Furthermore, the norm of its approximation error satisfies $\| \mathcal{X} \times_n ({\bf I} - {\bf P}) \|^2 = \| \mathcal{X} \|^2 - \| \mathcal{X} \times_n {\bf P} \|^2$.  By continuing to use similar ideas in combination with lemma~\ref{lem:modeProdProps} for all modes one can prove the following more general Pythagorean result (see, e.g., theorem 5.1 in \cite{vannieuwenhoven2012new}).

\begin{lem}
Let $\mathcal{X} \in \mathbb{R}^{I_1\times I_2\times ... \times I_N}$ and ${\bf U}^{(n)}\in \mathbb{R}^{I_n\times I_n}$ be an orthogonal projection matrix for all $n \in [N]$.  Then,
$$ \left\| \mathcal{X} - \mathcal{X}\times_1 {\bf U}^{(1)}\times_2 {\bf U}^{(2)}...\times_N {\bf U}^{(N)} \right\|^2 =: \left\| \mathcal{X} - \mathcal{X} \bigtimes_{n=1}^N {\bf U}^{(n)} \right\|^2 = \sum^N_{n=1} \left\| \mathcal{X} \bigtimes_{h=1}^{n-1} {\bf U}^{(h)} \times_{n} \left({\bf I} - {\bf U}^{(n)} \right) \right\|^2.$$
\label{lem:TensorPythagorean}
\end{lem}

\subsection{The Higher Order Singular Value Decomposition (HoSVD)}
 Any tensor $\mathcal{X}\in \mathbb{R}^{I_1\times I_2\times ...\times I_N}$ can be decomposed as mode products of a core tensor $\mathcal{C}\in \mathbb{R}^{I_1\times I_2\times ...\times I_N}$ with $N$ orthogonal matrices ${\bf U}^{(n)}\in \mathbb{R}^{I_n\times I_n}$ each of which is composed of the left singular vectors of ${\bf X}_{(n)}$ \cite{de2000multilinear}:
\begin{equation}
\mathcal{X} = \mathcal{C}\times_1 {\bf U}^{(1)} \times_2 {\bf U}^{(2)}...\times_N{\bf U}^{(N)}=\mathcal{C}\bigtimes_{n=1}^N {\bf U}^{(n)}
\end{equation}
\noindent where $\mathcal{C}$ is computed as 
\begin{equation}
\mathcal{C} = \mathcal{X}\times_1 \left({\bf U}^{(1)}\right)^\top \times_2 \left({\bf U}^{(2)}\right)^\top ...\times_N \left({\bf U}^{(N)}\right)^\top .
\end{equation}

Let $\mathcal{C}_{i_n=\alpha}$
be a subtensor of $\mathcal{C}$ obtained by fixing the $n$th index to $\alpha$. This subtensor satisfies the following properties:

\begin{itemize}
\item all-orthogonality: $\mathcal{C}_{i_n=\alpha} $ and $\mathcal{C}_{i_n=\beta}$ are orthogonal for all possible values of $n$, $\alpha$ and $\beta$ subject to $\alpha \neq \beta$.
\begin{equation}
\langle \mathcal{C}_{i_n=\alpha},\; \mathcal{C}_{i_n=\beta} \rangle = 0 \; when\; \alpha \neq \beta.
\end{equation}
\item  ordering: 
\begin{equation}
\parallel \mathcal{C}_{i_n=1} \parallel \geq \parallel \mathcal{C}_{i_n=2} \parallel\geq ... \geq \parallel \mathcal{C}_{i_n=I_n} \parallel\geq 0
\end{equation}

for $n\in \left[N\right]$.

\end{itemize}

\section{Multiscale Analysis of Higher-order Datasets}
\label{multiScaleMethod}

In this section, we present a new tensor decomposition method named Multiscale HoSVD (MS-HoSVD) for an $N$th order tensor, $\mathcal{X}\in \mathbb{R}^{I_1\times I_2\times ...\times I_N}$.    
The proposed method recursively applies the following two-step approach: (i) Low-rank tensor approximation, followed by (ii) Partitioning the residual (original minus low-rank) tensor into subtensors.   

A tensor $\mathcal{X}$ is first decomposed using HoSVD as follows:
\begin{equation}
\mathcal{X}= \mathcal{C}\times_1 {\bf U}^{(1)} \times_2 {\bf U}^{(2)}...\times_N{\bf U}^{(N)},
\end{equation}
\noindent where the ${\bf U}^{(n)}$'s are the left singular vectors of the unfoldings ${\bf X}_{(n)}$. The low-rank approximation of $\mathcal{X}$ is obtained by
\begin{equation}
\hat{\mathcal{X}}_0= \mathcal{C}_0\times_1 \hat{\bf U}^{(1)} \times_2 \hat{\bf U}^{(2)}...\times_N \hat{\bf U}^{(N)}
\label{equ:DefScale0Approx}
\end{equation}
\noindent where $\hat{\bf U}^{(n)}\in \mathbb{R}^{I_n\times r_n}$s  are the truncated matrices obtained by keeping the first $r_n$ columns of ${\bf U}^{(n)}$ and $\mathcal{C}_0=\mathcal{X}\times_1 \left( \hat{\bf U}^{(1)} \right)^{\top} \times_2 \left( \hat{\bf U}^{(2)} \right)^{\top}...\times_N \left( \hat{\bf U}^{(N)} \right)^{\top}$. The multilinear-rank of $\hat{\mathcal{X}}_0$, $\left\lbrace r_1,...,\;r_N \right\rbrace$, can either be given a \textit{priori}, or an energy criterion can be used to determine the minimum number of singular values to keep along each mode as:
\begin{equation}
r_n={\rm arg} \min_i \sum_{l=1}^i \sigma^{(n)}_l  \;\;\;  s.t. \;\;\; \dfrac{\sum_{l=1}^i \sigma^{(n)}_l }{ \sum_{l=1}^{I_{n}}\sigma^{(n)}_l} \geq \tau,
\label{equ:tauDef}
\end{equation}
\noindent where $\sigma^{(n)}_l$ is the $l$th singular value of the matrix obtained from the SVD of the unfolding ${\bf X}_{(n)}$, 
and $\tau$ is an energy threshold. Once $\hat{\mathcal{X}}_0$ is obtained, the tensor $\mathcal{X}$ can be written as 
\begin{equation}
\mathcal{X}= \hat{\mathcal{X}}_0 + \mathcal{W}_0,
\label{equ:ResidualDef}
\end{equation}
\noindent where $\mathcal{W}_0$ is the residual tensor.

For the first-scale analysis, to better encode the details of $\mathcal{X}$, we adapted an idea similar to the one presented in \cite{ozdemir2016multiscale, ozdemir2015locally}. The $0^{\rm th}$-scale residual tensor, $\mathcal{W}_0$ is first decomposed into subtensors as follows. $\mathcal{W}_0\in \mathbb{R}^{I_1\times I_2\times ...\times I_N}$ is unfolded across each mode yielding ${\bf W}_{0,(n)} \in \mathbb{R}^{I_n\times \prod_{j\neq n}I_j}$ whose columns are the mode-$n$ fibers of $\mathcal{W}_0$. 
For each mode, rows of ${\bf W}_{0,(n)}$ are partitioned into $c_n$ non-overlapping clusters using a clustering algorithm such as local subspace analysis (LSA) \cite{yan2006general} in order to encourage the formation of new subtensors which are intrinsically lower rank, and therefore better approximated via a smaller HoSVD at the next scale. The Cartesian product of the partitioning labels
coming from the $N$ modes yields $K=\prod_{i=1}^{N}c_{i}$ disjoint subtensors $\mathcal{X}_{1,k}$ where $k\in [K]$.

Let $ J_{0}^n $ be the index set corresponding to  the $n$th mode of  $\mathcal{W}_{0}$ with $J_0^n=[I_n]$, 
and let $ J_{1,k}^n $ be the index set of the subtensor $\mathcal{X}_{1,k}$ for the $n$th mode, where $J_{1,k}^n\subset J_0^n$ for all $k \in [K]$ and $n\in [N]$. Index sets of subtensors for the $n$th mode satisfy $\bigcup_{k=1}^K J_{1,k}^n= J_0^n$  for all $n \in [N]$. The $k$th subtensor $\mathcal{X}_{1,k} \in \mathbb{R}^{ \left| J_{1,k}^1 \right| \times \left| J_{1,k}^2 \right| \times \dots \times \left| J_{1,k}^N \right|}$ is obtained by 
\begin{equation}
\begin{array}{c}
\mathcal{X}_{1,k}(i_1,\;i_2,...,\;i_N)= \mathcal{W}_0(J_{1,k}^1(i_1),\; J_{1,k}^2(i_2), ...,\; J_{1,k}^N(i_N)),\\
\mathcal{X}_{1,k}= \mathcal{W}_0(J_{1,k}^1\times J_{1,k}^2 \times ...\times J_{1,k}^N), 
\end{array}
\label{equ:subtensor1}
\end{equation} 
\noindent where $i_n \in \left[ \left| J_{1,k}^n \right| \right]$. Low-rank approximation for each subtensor is obtained by applying HoSVD as:

\begin{equation}
\hat{\mathcal{X}}_{1,k} = \mathcal{C}_{1,k}\times_1 \hat{\bf U}^{(1)}_{1,k} \times_2 \hat{\bf U}^{(2)}_{1,k}...\times_N \hat{\bf U}^{(N)}_{1,k},
\label{equ:LowrankSubTensor}
\end{equation} 
 
\noindent where $\mathcal{C}_{1,k}$ and $\hat{\bf U}^{(n)}_{1,k} \in \mathbb{R}^{|J^{n}_{1,k}|\times r^{(n)}_{1,k}}$s correspond to the core tensor and low-rank projection basis matrices of $\mathcal{X}_{1,k}$, respectively. We can then define $\hat{\mathcal{X}}_1$ as the $1^{\rm st}$-scale approximation of $\mathcal{X}$ formed by  mapping all of the subtensors onto $\hat{\mathcal{X}}_{1,k}$ as follows:
\begin{equation}
\hat{\mathcal{X}}_1(J_{1,k}^1\times J_{1,k}^2 \times ...\times J_{1,k}^n)=\hat{\mathcal{X}}_{1,k}.
\label{equ:FirstScaleApproxDef}
\end{equation}
Similarly, $1^{\rm st}$ scale residual tensor is obtained by
\begin{equation}
{\mathcal{W}}_1(J_{1,k}^1\times J_{1,k}^2 \times ...\times J_{1,k}^n)=\mathcal{W}_{1,k},
\end{equation}
\noindent where $\mathcal{W}_{1,k} = \mathcal{X}_{1,k} -\hat{\mathcal{X}}_{1,k}$. Therefore, $\mathcal{X}$ can be rewritten as:
\begin{equation}
\mathcal{X}=\hat{\mathcal{X}}_0+\mathcal{W}_0= \hat{\mathcal{X}}_0+ \hat{\mathcal{X}}_1 + \mathcal{W}_1.
\label{equ:FirstScaleResidualDef}
\end{equation}

Continuing in this fashion the $j^{th}$ scale approximation of $\mathcal{X}$ is obtained by partitioning $\mathcal{W}_{j-1,k}$s into subtensors $\mathcal{X}_{j,k}$s and fitting a low-rank model to each one of them in a similar fashion. Finally, the $j^{th}$ scale decomposition of $\mathcal{X}$ can be written as: 
\begin{equation}
\mathcal{X}= \sum_{i=0}^j{\hat{\mathcal{X}}_i} + \mathcal{W}_j.
\end{equation}
Algorithm \ref{mshosvd_algo} describes the pseudo code for this approach and Fig. \ref{1_scale_MS-HoSVD} illustrates 1-scale MS-HoSVD.

\begin{figure}[H]
\centering
\includegraphics[width=17 cm]{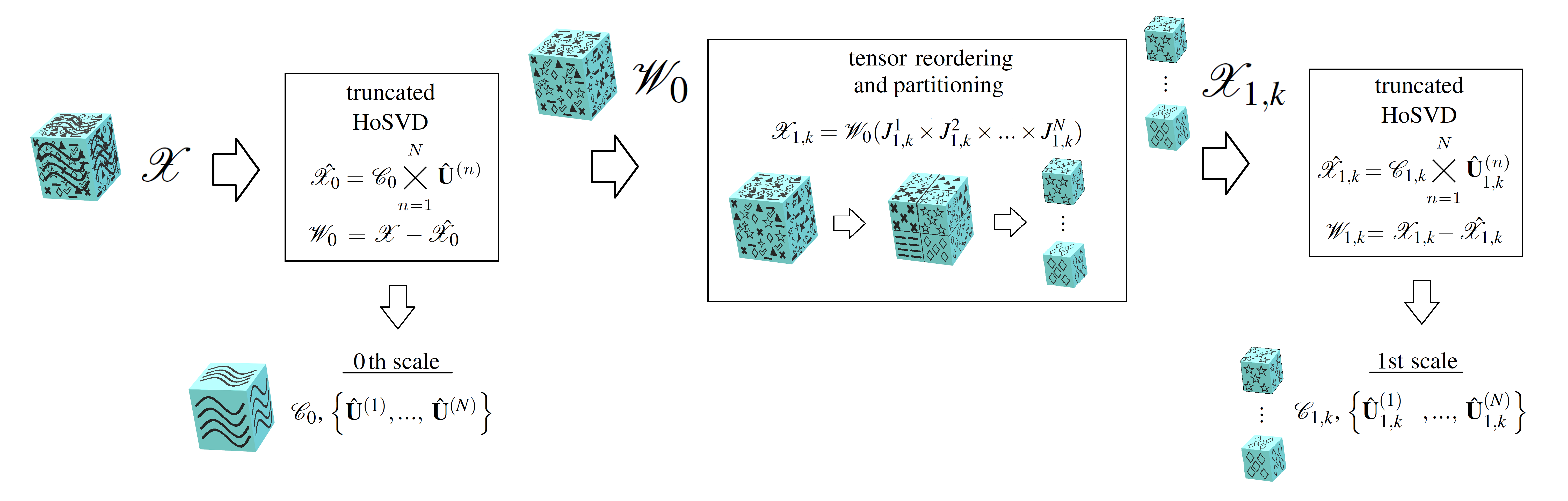}
\caption{Illustration of 1-scale MS-HoSVD. Higher scale decomposition can be obtained by applying the illustrated approach to the residual tensors recursively.}
\label{1_scale_MS-HoSVD}
\end{figure}

\begin{algorithm}[h]
\small
\caption{Multiscale HoSVD}
\label{mshosvd_algo}
\begin{algorithmic}[1]
\STATE Input:  ${\mathcal{X}}$: tensor ,  ${\bf C}= \left( c_1,\; c_2,\; ...,\; c_N\right)$: the desired number of clusters for each mode, $s_H$: the highest scale of MS-HoSVD. 
\STATE Output:  $\mathit{T}$: Tree structure containing the MS-HoSVD decomposition of $\hat{\mathcal{X}}$.

\STATE Create an empty tree $\mathit{T}$
\STATE Create an empty list $\mathit{L}$
\STATE Add the node containing ${\mathcal{X}} =: \mathcal{X}_{0,1}$ to $\mathit{L}$ with ${\rm Parent}(0,1) = \emptyset$ (i.e., this is the root of the the tree).

\WHILE {$\mathit{L}$ is not empty.}
\STATE Pop a node corresponding to ${\mathcal{X}}_{s,t}$ (the $t$th subtensor from $s$th scale) from the list $\mathit{L}$  where $s\in \left\lbrace 0,..., s_H\right\rbrace$ and $t\in \left\lbrace 1,..., K^s\right\rbrace$.
\STATE $\mathcal{C}_{s,t}$, $\left\lbrace \hat{\bf U}^{(n)}_{s,t}\right\rbrace \leftarrow $ truncatedHOSVD($\mathcal{X}_{s,t}$).

\STATE Add the node containing $\mathcal{C}_{s,t}$, $\left\lbrace \hat{\bf U}^{(n)}_{s,t}\right\rbrace$ to $\mathit{T}$ as a child of ${\rm Parent}(s,t)$.
\IF {$s<s_H$}
\STATE Compute $\mathcal{W}_{s,t}={\mathcal{X}}_{s,t}-\hat{\mathcal{X}}_{s,t}$.
\STATE Create $K$ subtensors $\mathcal{X}_{s+1,K(t-1)+k}$  with $J_{s+1,K(t-1)+k}^n$ from $\mathcal{W}_{s,t}$ where $k\in \left\lbrace 1,\;2,\;...,\;K \right\rbrace$ and $n\in \left\lbrace 1,\;2,\;...,\;N \right\rbrace$. 
\STATE Add $K$ nodes containing $\mathcal{X}_{s+1,K(t-1)+k}$  and $\left\lbrace J_{s+1,K(t-1)+k}^n\right\rbrace$ to  $\mathit{L}$ with ${\rm Parent}(s+1,K(t-1)+k) = (s,t)$.
\ENDIF
\ENDWHILE

\end{algorithmic}
\end{algorithm}

\subsection{Memory Cost of the First Scale Decomposition}
\label{fixedRankAnalysis}
Let $\mathcal{X}\in \mathbb{R}^{I_1\times I_2 \times .... \times I_N}$ be an $N$th order tensor. To simplify the notation, assume that the dimension of each mode is the same, i.e. $I_1 = I_2 = .... = I_N=I$. Assume  $\mathcal{X}$  is approximated by HoSVD as: 

\begin{equation}
\hat{\mathcal{X}} = \mathcal{C}_H\times_1 {\bf U}^{(1)}_H \times_2 {\bf U}^{(2)}_H...\times_N{\bf U}^{(N)}_H,
\end{equation}

 \noindent by fixing the rank of each mode matrix as $\mbox{rank}({\bf U}^{(i)}_H) = r_H$ for $i\in\left\lbrace 1,\;2,...,\; N\right\rbrace$. Let $\mathbb{F}(\cdot)$ be a function that quantifies the memory cost,  then the storage cost of ${\mathcal{X}}$ decomposed by HoSVD is $\mathbb{F}(\mathcal{C}_H) + \sum_{i=1}^N(\mathbb{F}({\bf U}^{(i)}_H))\approx r_H^N + N I r_H$. 

For multiscale analysis at scale 1, $\hat{\mathcal{X}}= \hat{\mathcal{X}}_0 + \hat{\mathcal{X}}_1 $.
The cost of storing $\hat{\mathcal{X}}_0$ is $\mathbb{F}(\mathcal{C}_0) + \sum_{i=1}^N(\mathbb{F}(\hat{\bf U}^{(i)})) \approx r_0^N + N I r_0$ where the rank of each mode matrix is fixed at $\mbox{rank}({\bf U}^{(i)}) = r_0$ for $i\in\left\lbrace 1,\;2,...,\; N\right\rbrace$. The cost of storing $\hat{\mathcal{X}}_1$ is the sum of the storage costs for each of the $K=\prod _{i=1}^N c(i)$ subtensors $\hat{\mathcal{X}}_{1,k}$. Assume $c(i)=c$ for all $i\in\left\lbrace 1,\;2,...,\; N\right\rbrace$ yielding $c^N$ equally sized subtensors, and that each $\hat{\mathcal{X}}_{1,k}$ is decomposed using the HoSVD as $\hat{\mathcal{X}}_{1,k} = \mathcal{C}_{1,k}\times_1 \hat{\bf U}^{(1)}_{1,k} \times_2 \hat{\bf U}^{(2)}_{1,k}...\times_N \hat{\bf U}^{(N)}_{1,k}$. Let the rank of each mode matrix be fixed as $\mbox{rank}(\hat{\bf U}^{(i)}_{1,k}) = r_1$ for all $i\in\left\lbrace 1,\;2,...,\; N\right\rbrace$ and $k\in\left\lbrace 1,\;2,...,\; K\right\rbrace$. Then, the memory cost for the first scale is $\sum_{k=1}^K \left(\mathbb{F}(\mathcal{C}_{1,k})+\sum_{i=1}^N\mathbb{F}(\hat{\bf U}^{(i)}_{1,k}) \right) \approx c^N \left(r^N_1+ \dfrac{N I r_1}{c}  \right)$. Choosing $r_1 \lesssim \dfrac{r_0}{c^{(N-1)}}$ ensures that the storage cost does not grow exponentially so that $\mathbb{F}(\hat{\mathcal{X}}_1)<\mathbb{F}(\hat{\mathcal{X}}_0)$ since the total cost becomes approximately equal to $r_0^N \left(1+\dfrac{1}{c^{N^2-2N}} \right)+2N I  r_0$. Thus, picking $r_0 \approx r_H/2$ can now provide lower storage cost for the first scale analysis than for HoSVD.

\subsection{Computational Complexity}
The computational complexity of MS-HoSVD at the first scale is equal to the sum of computational complexity of computing HoSVD at the parent node, partitioning into subtensors and computing HoSVD for each one of the subtensors.
Computational complexity of HoSVD of an N-way tensor $\mathcal{X}\in \mathbb{R}^{I_1\times I_2\times ...\times I_N}$ where ${I_1= I_2= ...= I_N =I}$ is $\mathcal{O}\left(N I^{(N+1)}\right)$ \cite{karami2012compression}.  By assuming that the partitioning is performed using K-means (via Lloyd's algorithm) with $c_i=c$ along each mode, the complexity partitioning along each mode is $\mathcal{O} \left( N   {I}^N  c  i \right)$, where $i$ is the number of iterations used in Lloyd's algorithm.  Finally, the total complexity of applying the HoSVD to $c^N$ equally sized subtensors is $\mathcal{O}\left(c^N  N {\left(I/c \right)}^{(N+1)}\right)$. 
Therefore, first scale MS-HoSVD has a total computational complexity of ${\mathcal O}\left(NI^{(N+1)} + N {I}^N  c  i + c^N N {\left(I/c \right)}^{(N+1)}\right) $.  Note that this complexity is similar to that of the HoSVD whenever $c  i$ is small compared to $I$. The runtime complexity of these multiscale methods can be reduced even further by  computing the HoSVDs for different subtensors in parallel whenever possible, as well as by utilizing distributed and parallel SVD algorithms such as \cite{iwen2016distributed} when computing all the required HoSVD decompositions.

\subsection{A Linear Algebraic Representation of the Proposed Multiscale HoSVD Approach}

Though the tree-based representation of the proposed MS-HoSVD approach used above in Algorithm~\ref{mshosvd_algo} is useful for algorithmic development, it is somewhat less useful for theoretical error analysis.  In this subsection we will develop formulas for the proposed MS-HoSVD approach which are more amenable to error analysis.  In the process, we will also formulate a criterion which, when satisfied, guarantees that the proposed fist scale MS-HoSVD approach produces an accurate multiscale approximation to a given tensor.

\noindent {\bf Preliminaries:}  We can construct full size first-scale subtensors of the residual tensor $\mathcal{W}_0 \in \mathbb{R}^ {I_1 \times I_2 \times ... I_N }$ from \eqref{equ:ResidualDef}, $\mathcal{X}|_k \in \mathbb{R}^ {I_1 \times I_2 \times ... I_N }$ for all $k \in [K]$, using the index sets $J_{1,k}^n$ from \eqref{equ:subtensor1} along with diagonal restriction matrices.  Let ${\bf R}_{k}^{(n)}\in \left\lbrace 0,\;1 \right\rbrace^{I_n\times  I_n}$ be the diagonal matrix with entries given by 
\begin{equation}
{\bf R}_{k}^{(n)}(i,j)= \begin{cases}
    1,& \text{if } i=j,\; \mbox{and} \;  j \in J_{1,k}^n \\        0,         & \text{otherwise}
\end{cases}
\label{equ:RestrictionMat}
\end{equation}
for all $k \in [K]$, and $n \in [N]$.  We then define
\begin{equation}
 \mathcal{X}|_k := \mathcal{W}_0 \bigtimes_{n=1}^N {\bf R}_{k}^{(n)} = \mathcal{W}_0 \times_1 {\bf R}_{k}^{(1)} \times_2 {\bf R}_{k}^{(2)}...\times_N {\bf R}_{k}^{(N)}.
\label{equ:subTensorDef}
\end{equation} 
Thus, the $k$th subtensor $\mathcal{X}|_k$ will only have nonzero entries, given by $\mathcal{W}_0(J_{1,k}^1\times...\times J_{1,k}^N)$, in the locations indexed by the sets $J_{1,k}^n$ from above.  The properties of the index sets $J_{1,k}^n$ furthermore guarantee that these subtensors all have disjoint support.  As a result both
\begin{equation}
\mathcal{W}_0 = \sum^K_{k = 1} \mathcal{X}|_k
\label{equ:SubTensorAddtoBigTensor}
\end{equation}
and
$$\left\langle \mathcal{X}|_k,\mathcal{X}|_j \right\rangle = 0 ~\textrm{for all}~ j,k \in [K] ~\textrm{with}~ j \neq k$$
will always hold.

Recall that we want to compute the HoSVD of the subtensors we form at each scale in order to create low-rank projection basis matrices along the lines of those in \eqref{equ:LowrankSubTensor}.  Toward this end we compute the top $r^{(n)}_{k}\leq rank({\bf R}_{k}^{(n)}) = |J^{n}_{1,k}|$ left singular vectors of the mode-n matricization of each $\mathcal{X}|_k$, ${\bf X|_k}_{(n)} \in \mathbb{R}^{I_n \times \prod_{m \neq n} I_m}$, for all $n \in [N]$.  Note that ${\bf X|_k}_{(n)} = {\bf R}_{k}^{(n)} {\bf X|_k}_{(n)}$ always holds for these matricizations since ${\bf R}_{k}^{(n)}$ is a projection matrix.\footnote{Here we are implicitly using \eqref{equ:KronModenFlat}.}  Thus, the top $r^{(n)}_{k}$ left singular vectors of ${\bf X|_k}_{(n)}$ will only have nonzero entries in locations indexed by   $J^{n}_{1,k}$.  Let $\hat{\bf U}^{(n)}_{k} \in \mathbb{R}^{I_n\times r^{(n)}_{k}}$ be the matrix whose columns are these top singular vectors.  As a result of the preceding discussion we can see that $\hat{\bf U}^{(n)}_{k} = {\bf R}_{k}^{(n)} \hat{\bf U}^{(n)}_{k}$ will hold for all $n \in [N]$ and $k \in [K]$.
Our low-rank projection matrices ${\bf Q}^{(n)}_k\in \mathbb{R}^{I_n\times  I_n}$ used to produce low-rank approximations of each subtensor $\mathcal{X}|_k$ can now be defined as 
\begin{equation}
{\bf Q}^{(n)}_k := \hat{\bf U}^{(n)}_{k} \left( \hat{\bf U}^{(n)}_{k} \right)^\top.
\label{equ:Scale1SubTensorProjMats}
\end{equation}

As a consequence of $\hat{\bf U}^{(n)}_{k} = {\bf R}_{k}^{(n)} \hat{\bf U}^{(n)}_{k}$ holding, combined with the fact that $\left( {\bf R}_{k}^{(n)} \right)^\top = {\bf R}_{k}^{(n)}$ since each ${\bf R}_{k}^{(n)}$ matrix is diagonal, we have that
\begin{equation}
{\bf Q}^{(n)}_k := \hat{\bf U}^{(n)}_{k} \left( \hat{\bf U}^{(n)}_{k} \right)^\top =  {\bf R}_{k}^{(n)}\hat{\bf U}^{(n)}_{k} \left( {\bf R}_{k}^{(n)}\hat{\bf U}^{(n)}_{k} \right)^\top = {\bf R}_{k}^{(n)}\hat{\bf U}^{(n)}_{k} \left( \hat{\bf U}^{(n)}_{k} \right)^\top \left( {\bf R}_{k}^{(n)} \right)^\top = {\bf R}_{k}^{(n)} {\bf Q}^{(n)}_k {\bf R}_{k}^{(n)}
\label{DefSubTensorProj}
\end{equation}
holds for all $n \in [N]$ and $k \in [K]$.  Using \eqref{DefSubTensorProj} combined with the fact that ${\bf R}_{k}^{(n)}$ is a projection matrix, we can further see that
\begin{equation}
\footnotesize
{\bf R}_{k}^{(n)} {\bf Q}^{(n)}_k = {\bf R}_{k}^{(n)} \left( {\bf R}_{k}^{(n)} {\bf Q}^{(n)}_k {\bf R}_{k}^{(n)} \right) =  {\bf R}_{k}^{(n)} {\bf Q}^{(n)}_k {\bf R}_{k}^{(n)}  = {\bf Q}^{(n)}_k = {\bf R}_{k}^{(n)} {\bf Q}^{(n)}_k {\bf R}_{k}^{(n)} = \left( {\bf R}_{k}^{(n)} {\bf Q}^{(n)}_k {\bf R}_{k}^{(n)} \right) {\bf R}_{k}^{(n)} = {\bf Q}^{(n)}_k {\bf R}_{k}^{(n)}
\label{equ:PropResProjStuff}
\end{equation}
also holds for all $n \in [N]$ and $k \in [K]$.\\ 

\noindent {\bf 1-scale Analysis of MS-HoSVD:} Using this linear algebraic formulation we are now able to re-express the the $1^{\rm st}$ scale approximation of $\mathcal{X} \in \mathbb{R}^ {I_1 \times I_2 \times ... I_N }$, $\hat{\mathcal{X}}_1 \in \mathbb{R}^ {I_1 \times I_2 \times ... I_N }$, as well as the $1^{\rm st}$ scale residual tensor tensor, ${\mathcal{W}}_1 \in \mathbb{R}^ {I_1 \times I_2 \times ... I_N }$, as follows (see \eqref{equ:FirstScaleApproxDef} -- \eqref{equ:FirstScaleResidualDef}).  We have that
\begin{align}
\hat{\mathcal{X}}_1 = \sum^K_{k = 1} \left( \mathcal{X}|_k \bigtimes_{n=1}^N {\bf Q}^{(n)}_k \right) & = \sum^K_{k = 1} \left( \mathcal{W}_0 \bigtimes_{n=1}^N {\bf Q}^{(n)}_k {\bf R}_{k}^{(n)} \right) &\left(\textrm{Using Lemma~\ref{lem:modeProdProps} and \eqref{equ:subTensorDef}} \right) \nonumber\\
& = \sum^K_{k = 1} \left( \mathcal{W}_0 \bigtimes_{n=1}^N {\bf Q}^{(n)}_k \right) & \left(\textrm{Using the properties in \eqref{equ:PropResProjStuff}} \right) \nonumber\\
& = \sum^K_{k = 1} \left( \left( \mathcal{X} - \hat{\mathcal{X}}_0 \right) \bigtimes_{n=1}^N {\bf Q}^{(n)}_k \right) & \left(\textrm{Using \eqref{equ:ResidualDef}} \right) 
\label{equ:Scale1ApproxDef}
\end{align}
holds.  Thus, we see that the residual error $\mathcal{W}_1$ from \eqref{equ:FirstScaleResidualDef} satisfies
\begin{equation}
\mathcal{X} = \hat{\mathcal{X}}_0+ \sum^K_{k = 1} \left( \left( \mathcal{X} - \hat{\mathcal{X}}_0 \right) \bigtimes_{n=1}^N {\bf Q}^{(n)}_k \right) + \mathcal{W}_1.
\label{equ:ErrorFormulaForScale1}
\end{equation}

Having derived \eqref{equ:ErrorFormulaForScale1}, it behooves us to consider when using such a first-scale approximation of $\mathcal{X}$ is actually better than, e.g., just using a standard HoSVD-based $0^{\rm th}$-scale approximation of $\mathcal{X}$ along the lines of \eqref{equ:ResidualDef}.  As one might expect, this depends entirely on $(i)$ how well the $1^{\rm st}$-scale partitions (i.e., the restriction matrices utilized in \eqref{equ:RestrictionMat}) are chosen, as well as on $(ii)$ how well restriction matrices of the type used in \eqref{equ:RestrictionMat} interact with the projection matrices used to create the standard HoSVD-based approximation in question.  Toward understanding these two conditions better, recall that $\hat{\mathcal{X}}_0 \in \mathbb{R}^ {I_1 \times I_2 \times ... I_N }$ in \eqref{equ:ErrorFormulaForScale1} is defined as 
\begin{equation}
\hat{\mathcal{X}}_0 = \mathcal{X} \bigtimes_{n=1}^N {\bf P}^{(n)} = \mathcal{X} \times_1 {\bf P}^{(1)} \times_2 {\bf P}^{(2)} \dots \times_N {\bf P}^{(N)}
\label{equ:LittleHoSVDapprox}
\end{equation}
where the orthogonal projection matrices ${\bf P}^{(n)} \in \mathbb{R}^{I_n \times I_n}$ are given by ${\bf P}^{(n)} = \hat{\bf U}^{(n)} \left( \hat{\bf U}^{(n)} \right)^\top$ for the matrices $\hat{\bf U}^{(n)}\in \mathbb{R}^{I_n\times r_n}$ used in \eqref{equ:DefScale0Approx}.  For simplicity let the ranks of the ${\bf P}^{(n)}$ projection matrices momentarily satisfy $r_1 = r_2 = \dots = r_N =: r_0$ (i.e., let them all be rank $r_0 < \max_n \{ {\rm rank}({\bf X}_{(n)}) \}$).  Similarly, let all 
the ranks, $r^{(n)}_{k}$, of the $1^{\rm st}$ scale projection matrices ${\bf Q}^{(n)}_k$ in \eqref{equ:Scale1SubTensorProjMats} be $r_1$ for the time being.  

Motivated by, e.g., the memory cost analysis of Section \ref{fixedRankAnalysis} above, one can now ask when the multiscale approximation error, $\| \mathcal{W}_1 \|$, resulting from \eqref{equ:ErrorFormulaForScale1} will be less than a standard HoSVD-based approximation error, $\| \mathcal{X} - \bar{\mathcal{X}}_0 \|$, where
\begin{equation}
\bar{\mathcal{X}}_0 := \mathcal{X} \bigtimes_{n=1}^N \bar{{\bf P}}^{(n)} = \mathcal{X} \times_1 \bar{\bf P}^{(1)} \times_2 \bar{\bf P}^{(2)} \dots \times_N \bar{\bf P}^{(N)},
\label{equ:BigHoSVDapprox}
\end{equation}
and each orthogonal projection matrix $\bar{{\bf P}}^{(n)}$ is of rank $\bar{r}_n = r_H \geq 2 r_0 \geq r_0 + c^{N-1}r_1$ \big(i.e., where each $\bar{{\bf P}}^{(n)}$ projects onto the top $r_H$ left singular vectors of ${\bf X}_{(n)}$\big).  In this situation having both $\| \mathcal{W}_1 \| < \| \mathcal{X} - \bar{\mathcal{X}}_0 \|$ and $r_H \geq 2 r_0 \geq r_0 + c^{N-1}r_1$ hold at the same time would imply that one could achieve smaller approximation error using MS-HoSVD than using HoSVD while simultaneously achieving better compression (recall Section \ref{fixedRankAnalysis}).  In order to help facilitate such analysis we prove error bounds in Appendix~\ref{sec:AppendixErrorBound} that are implied by the choice of a good partitioning scheme for the residual tensor $\mathcal{W}_0$ in \eqref{equ:RestrictionMat} -- \eqref{equ:SubTensorAddtoBigTensor}.

In particular, with respect to the question concerning how well the ${1}^{\rm st}$-scale approximation error, $\| \mathcal{W}_1 \|$, from \eqref{equ:ErrorFormulaForScale1} might compare to the HoSVD-based approximation error $\| \mathcal{X} - \bar{\mathcal{X}}_0 \|$, we can use the following notion of an \textit{effective partition of $\mathcal{W}_0$}. The partition of $\mathcal{W}_0$ formed by the restriction matrices ${\bf R}_{k}^{(n)}$ in \eqref{equ:RestrictionMat} -- \eqref{equ:SubTensorAddtoBigTensor} will be called \textit{effective} if there exists another \textit{pessimistic} partitioning of $\mathcal{W}_0$ via (potentially different) restriction matrices $\left\lbrace \tilde{\bf R}_{k}^{(n)}\right\rbrace^K_{k=1}$ together with a bijection $f:[K]\rightarrow[K]$ such that
\begin{equation}
\sum_{n=1}^N \left\| {\mathcal{X}}|_k\times_n \left({\bf I} - {\bf Q}^{(n)}_k \right) \right\|^2\leq \sum_{n=1}^N \left\| \mathcal{W}_0 \times_n \tilde{\bf R}_{f(k)}^{(n)} \left({\bf I}-\tilde{\bf P}^{(n)} \right)\bigtimes_{h\neq n}^N\tilde{\bf R}_{f(k)}^{(h)}\right\|^2
\label{equ:ExampleErrorCond}
\end{equation}
\noindent  holds for each $k\in [K]$.  In \eqref{equ:ExampleErrorCond} the $\left\lbrace \tilde{\bf P}^{(n)}\right\rbrace$ are the orthogonal projection matrices obtained from the HoSVD of $\mathcal{W}_0$ with ranks $\tilde{r}_n = r_H$ \big(i.e., where each $\tilde{{\bf P}}^{(n)}$ projects onto the top $\tilde{r}_n = r_H$ left singular vectors of the matricization ${\bf W}_{0,(n)}$\big).  
In Appendix~\ref{sec:AppendixErrorBound}, we  show that \eqref{equ:ExampleErrorCond} holding for $\mathcal{W}_0$ implies that the error $\| \mathcal{W}_1 \|$ resulting from our $1^{\rm st}$-scale approximation in \eqref{equ:ErrorFormulaForScale1} is less than an upper bound of the type often used for HoSVD-based approximation errors of the form $\| \mathcal{X} - \bar{\mathcal{X}}_0 \|$ (see, e.g., \cite{vannieuwenhoven2012new}).  In particular, we prove the following result.

\begin{thm}
Suppose that \eqref{equ:ExampleErrorCond} holds.  Then, the first scale approximation error given by MS-HoSVD in \eqref{equ:ErrorFormulaForScale1} is bounded by
\begin{align*}
\left\| \mathcal{W}_1 \right\|^2 = \left\| \mathcal{X}-\hat{\mathcal{X}}_0-\hat{\mathcal{X}}_1 \right\|^2 \leq \sum_{n=1}^N \left\| \mathcal{X}\times_n \left( {\bf I}-\bar{\bf P}^{(n)} \right) \right\|^2,
\end{align*}
where $\left\lbrace\bar{\bf P}^{(n)}\right\rbrace$ are low-rank  projection matrices of rank $\bar{r}_n = \tilde{r}_n = r_H$ obtained from the truncated HoSVD of $\mathcal{X}$ as per \eqref{equ:BigHoSVDapprox}.
\label{MainThm}
\end{thm}

\begin{proof}
See Appendix~\ref{sec:AppendixErrorBound}.
\end{proof}

Theorem~\ref{MainThm} implies that $\left\| \mathcal{W}_1 \right\|$ {\it may be less} than $\| \mathcal{X} - \bar{\mathcal{X}}_0 \|$ when \eqref{equ:ExampleErrorCond} holds.  It does not, however, actually prove that $\left\| \mathcal{W}_1 \right\| \leq \| \mathcal{X} - \bar{\mathcal{X}}_0 \|$ holds whenever \eqref{equ:ExampleErrorCond} does.  In fact, directly proving that $\left\| \mathcal{W}_1 \right\| \leq \| \mathcal{X} - \bar{\mathcal{X}}_0 \|$ whenever \eqref{equ:ExampleErrorCond} holds does not appear to be easy.  It also does not appear to be easy to prove the error bound in theorem~\ref{MainThm} without an assumption along the lines of \eqref{equ:ExampleErrorCond} which simultaneously controls both $(i)$ how well the restriction matrices utilized to partition $\mathcal{W}_0$ in \eqref{equ:subTensorDef} are chosen, as well as $(ii)$ how poorly (worst case) restriction matrices interact with the projection matrices used to create standard HoSVD-based approximations of $\mathcal{W}_0$ and/or $\mathcal{X}$.  The development of simpler and/or weaker conditions than \eqref{equ:ExampleErrorCond} which still yield meaningful error guarantees along the lines of theorem~\ref{MainThm} is left as future work.  See Appendix~\ref{sec:AppendixErrorBound} for additional details and comments, and Appendix~\ref{app:Example} below for an example illustrating Theorem~\ref{MainThm} on an idealized tensor..

\subsection{Adaptive Pruning in Multiscale HoSVD for Improved Performance}

In order to better capture the local structure of the tensor, it is important to look at higher scale decompositions. However, as the scale increases, the storage cost and computational complexity will increase making any gain in reconstruction error potentially not worth the additional memory cost. For this reason, it is important to carefully select the subtensors adaptively at higher scales. To help avoid the redundancy in decomposition structure we propose an adaptive pruning method across scales.

In adaptive pruning, the tree is pruned by minimizing the following cost function $\mathbb{H} = Error + \lambda \cdot Compression$ similar to the rate-distortion criterion commonly used by compression algorithms where $\lambda$ is the trade-off parameter \cite{ramchandran1993best}. To minimize this function we employ a greedy procedure similar to sequential forward selection \cite{pudil1994floating}. First, the root node which stores $\hat{\mathcal{X}}_0$ is created  and scale-1 subtensors $\hat{\mathcal{X}}_{1,k}$ are obtained from the 0th order residual tensor $\hat{\mathcal{W}}_0$ as discussed in Section \ref{multiScaleMethod}. These subtensors are stored in a list and the subtensor which decreases the cost function the most is then added to the tree structure under its parent node. Next, scale-2 subtensors belonging to the added node are created and added to the list. All of the scale-1 and scale-2 subtensors  in the list are again evaluated to find the subtensor that minimizes the cost function. This procedure is repeated until the cost function $\mathbb{H}$ converges or the decrease is minimal.  A pseudocode of the algorithm is given in Algorithm \ref{mshosvd_p_algo}. It is important to note that this algorithm is suboptimal similar to other greedy search methods.

\begin{algorithm}[h]
\small
\caption{Multiscale HoSVD with Adaptive Pruning}
\label{mshosvd_p_algo}
\begin{algorithmic}[1]
\STATE Input:  ${\mathcal{X}}$: tensor ,  ${\bf C}= \left( c_1,\; c_2,\; ...,\; c_N\right)$: the desired number of clusters for each modes, $s_H$: the highest scale of MS-HoSVD. 
\STATE Output:  $\mathit{T}$: Tree structure containing the MS-HoSVD decomposition of $\hat{\mathcal{X}}$.
\STATE Create an empty tree $\mathit{T}$.
\STATE Create an empty list $\mathit{L}$.
\STATE Add node containing ${\mathcal{X}}$ to $\mathit{L}$.




\WHILE {There is a node in $\mathit{L}$ that decreases the cost function $\mathbb{H}(\mathit{T})$.}
\STATE Find the node corresponding to ${\mathcal{X}}_{s,t}$ (the $t$th subtensor from $s$th scale) in the list $\mathit{L}$   that decreases $\mathbb{H}$ the  most where $s\in \left\lbrace 0,..., s_H\right\rbrace$ and $t\in \left\lbrace 1,..., K^{s}\right\rbrace$.
\STATE $\mathcal{C}_{s,t}$, $\left\lbrace \hat{\bf U}^{(n)}_{s,t}\right\rbrace \leftarrow $ truncatedHOSVD($\mathcal{X}_{s,t}$).

\STATE Add the node containing $\mathcal{C}_{s,t}$, $\left\lbrace \hat{\bf U}^{(n)}_{s,t}\right\rbrace$ to $\mathit{T}$.
\IF {$s<s_H$}
\STATE Compute $\mathcal{W}_{s,t}={\mathcal{X}}_{s,t}-\hat{\mathcal{X}}_{s,t}$.
\STATE Create $K$ subtensors $\mathcal{X}_{s+1,K(t-1)+k}$  with $J_{s+1,K(t-1)+k}^n$ from $\mathcal{W}_{s,t}$ where $k\in \left\lbrace 1,\;2,\;...,\;K \right\rbrace$ and $n\in \left\lbrace 1,\;2,\;...,\;N \right\rbrace$. 
\STATE Add $K$ nodes containing $\mathcal{X}_{s+1,K(t-1)+k}$  and $\left\lbrace J_{s+1,K(t-1)+k}^n\right\rbrace$ to  $\mathit{L}$.
\ENDIF
\ENDWHILE
\end{algorithmic}
\end{algorithm}

\section{Data Reduction}
\label{sec:DataCompress}

In this section we demonstrate the performance of MS-HoSVD for tensor type data reduction on several real 3-mode and 4-mode datasets as compared with three other tensor decompositions: HoSVD, H-Tucker, and T-Train. The performance of tensor decomposition methods are evaluated in terms of reconstruction error and compression rate. In the tables and figures below the error rate refers to the normalized tensor approximation error $\frac{\| \mathcal{X} - \hat{\mathcal{X}} \|_F}{\| \mathcal{X} \|_F}$ and the compression rate is computed as $\frac{\# {\rm ~total~ bits~ to~ store~ } \hat{\mathcal{X}}}{\# {\rm ~total~ bits~ to~ store~ } \mathcal{X}}$.
Moreover, we show the performance of the proposed adaptive tree pruning strategy for data reduction. 


\subsection{Datasets}

\subsubsection{PIE dataset}
\label{pie-datadescription}
A 3-mode tensor $\mathcal{X}\in \mathbb{R}^{ 244\times 320\times 138}$ is created from PIE dataset \cite{sim2003cmu}. The tensor contains 138 images from 6 different yaw angles and varying illumination conditions collected from a subject where each image is  converted to gray scale. Fig. \ref{pieSample} illustrates the images from different frames of the PIE dataset.

\begin{figure}[H]
\centering
\includegraphics[width=8 cm]{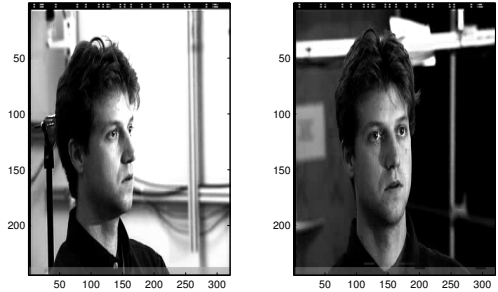}
\caption{Sample frames from PIE dataset corresponding to the 30th (left) and 80th (right) frames.}
\label{pieSample}
\end{figure}

\subsubsection{COIL-100 dataset}

The COIL-100 database contains 7200 images collected from 100 objects where the images of each object were taken at pose intervals of $5^{\circ}$. 
A 4-mode tensor $\mathcal{X}\in \mathbb{R}^{ 128\times 128\times 72 \times 100}$ is created from COIL-100 dataset \cite{nene1996columbia}. The constructed 4-mode tensor contains 72 images of size $128\times 128$ from 100 objects where each image is converted to gray scale. In Fig. \ref{coil100_sample}, sample images of four objects taken from different angles can be seen. 


\begin{figure}[H]
\centering
\includegraphics[width=7 cm]{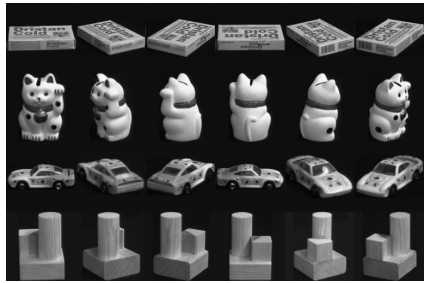}
\caption{Image samples of four different objects from COIL-100 dataset from varying pose angles (from $0^{\circ}$ to $240^{\circ}$ with $60^{\circ}$ increments).  }
\label{coil100_sample}
\end{figure}

\subsubsection{The Cambridge Hand Gesture Dataset}
The Cambridge hand gesture database consists of 
900 image sequences of nine gesture classes of three primitive hand shapes and three primitive motions where each class contains 100 image sequences (5 different illuminations $\times$ 10
arbitrary motions $\times$ 2 subjects). In Fig. \ref{HG_sample}, sample image sequences collected for nine hand gestures can be seen. The created 4-mode tensor $\mathcal{X}\in \mathbb{R}^{ 60\times 80\times 30 \times 900}$ contains 900 image sequences of size $60\times 80 \times 30$ where each image is converted to gray scale. 


\begin{figure}[H]
\centering
\includegraphics[width=8 cm]{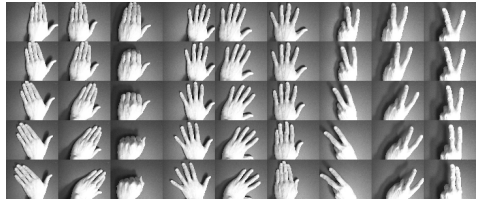}
\caption{Illustration of nine different classes in Cambridge Hand Gesture Dataset.}
\label{HG_sample}
\end{figure}

\subsection{Data Reduction Experiments}

In this section, we evaluate the performance of MS-HoSVD for 1 and 2-scale decompositions compared to HoSVD, H-Tucker and T-Train decompositions. In the following experiments, tensor partitioning is performed by LSA  and the cluster number along each mode is chosen as $c_i=2$. The rank used in HoSVD is  selected adaptively using the energy criterion as per Section \ref{multiScaleMethod}'s \eqref{equ:tauDef}. In our experiments, we performed MS-HoSVD with $\tau= 0.7$ and $\tau= 0.75$ and we kept $\tau$ the same for each scale. For the same compression rates as the MS-HoSVD, the reconstruction error of HoSVD, H-Tucker and T-Train models are computed.

Fig. \ref{data_reduction_bar} explores the interplay between compression rate and approximation error for MS-HoSVD in comparison to HoSVD, H-Tucker and T-Train for PIE, COIL-100 and hand gesture datasets. Starting from the left in Figs.  \ref{data_reduction_bar}(a), \ref{data_reduction_bar}(b) and \ref{data_reduction_bar}(c), the first two compression rates correspond to 1-scale MS-HoSVD with $\tau= 0.7$ and $\tau= 0.75$, respectively while the last two are obtained from 2-scale approximation with $\tau= 0.7$ and $\tau= 0.75$, respectively. As seen in Fig.  \ref{data_reduction_bar},  MS-HoSVD outperforms other approaches with respect to reducing  PIE, COIL-100 and hand gesture tensors at varying compression rates. Moreover, adding the $2^{\rm nd}$ scale  increases the storage requirements while decreasing the error of MS-HoSVD. Fig. \ref{piefig} illustrates the influence of scale on the visual quality of the reconstructed images.  As expected, introducing additional finer scales into a multiscale approximation of video data improves image detail in each frame. Moreover, the data reduction performance of T-Train is seen to be slightly better than H-Tucker in most of the experiments.

\begin{figure}[H]
\footnotesize{\centering}
\begin{minipage}[b]{0.48\linewidth}
\subfloat{\includegraphics[width=8 cm]{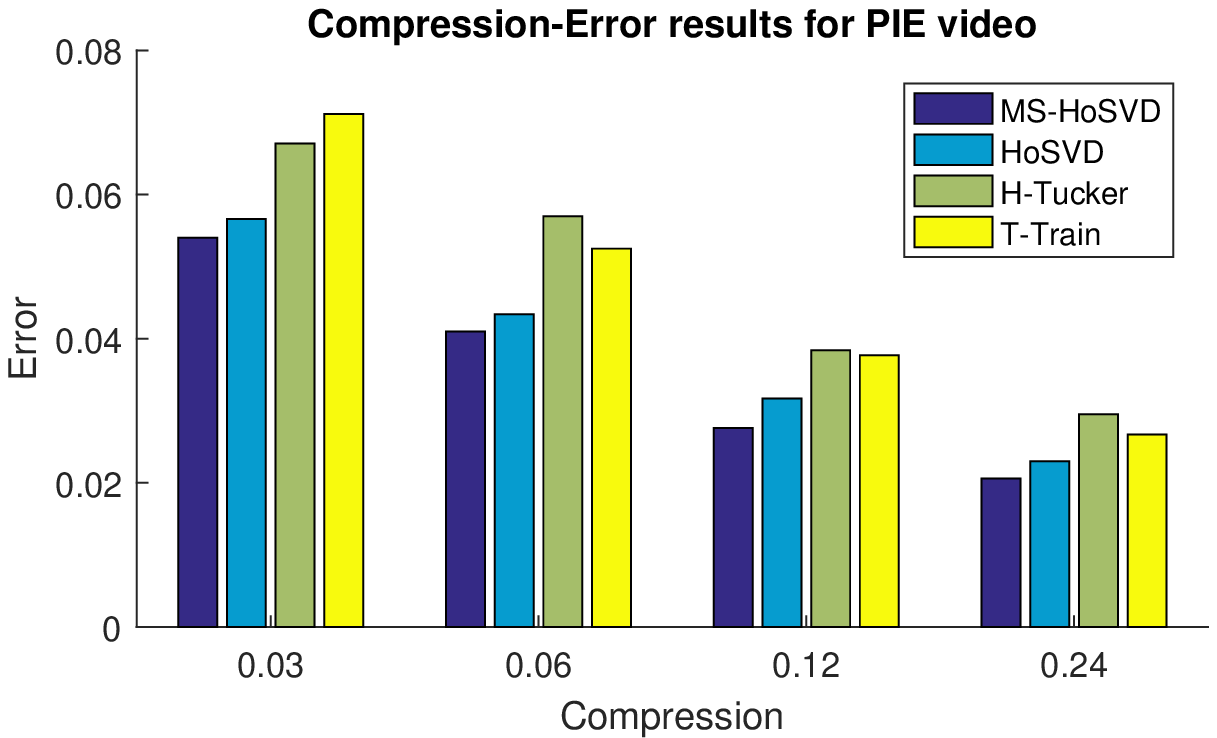}}\\
\centering (a) PIE \\
\end{minipage}
\hfill
\begin{minipage}[b]{0.48\linewidth}
\subfloat{\includegraphics[width=8 cm]{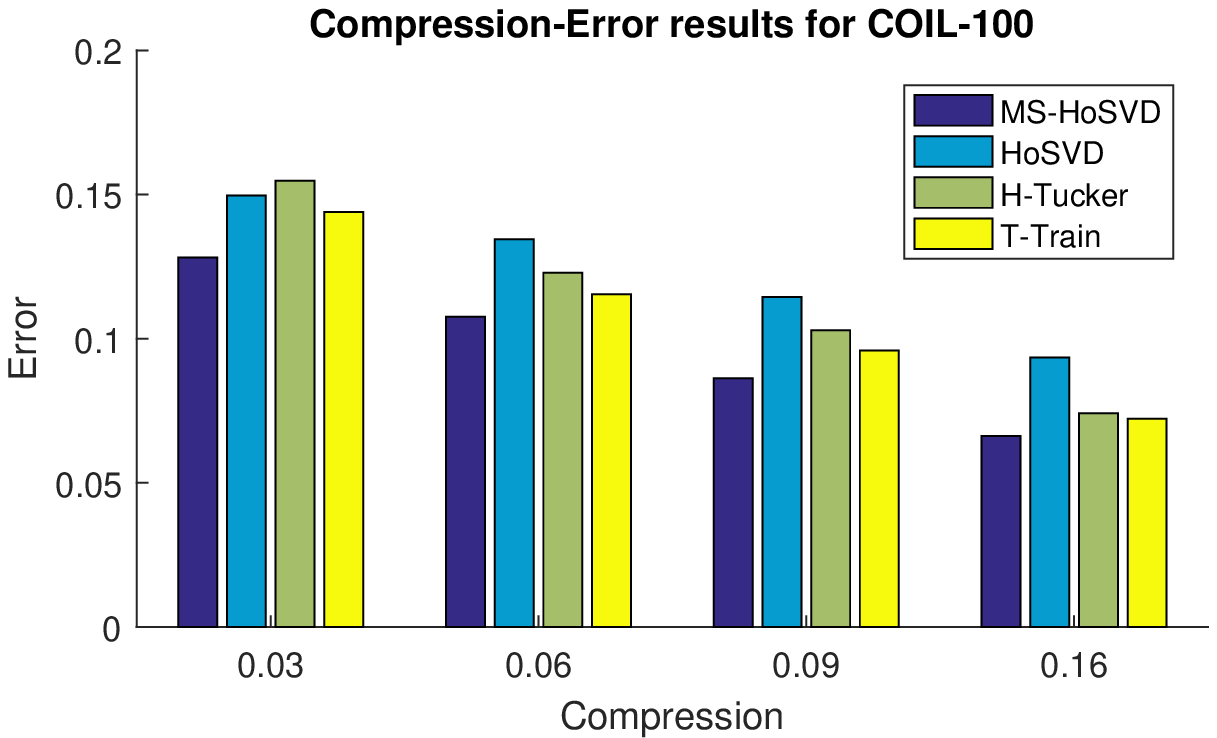}}\\
\centering (b) COIL-100\\
\end{minipage}
\\
\centering
\begin{minipage}[b]{0.48\linewidth}
\subfloat{\includegraphics[width=8 cm]{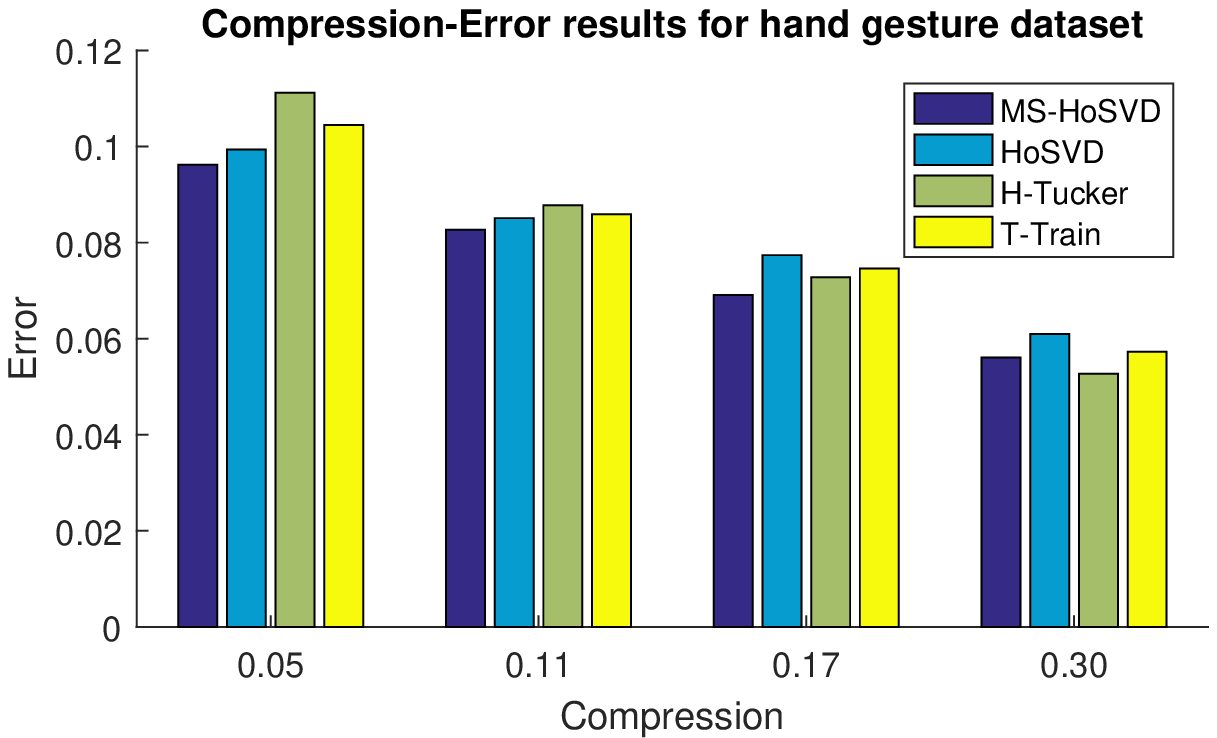}}\\
\centering (c) Hand Gesture\\
\end{minipage}



\caption{Compression rate versus Normalized Reconstruction Error  for MS-HoSVD (dark blue), HoSVD (light blue), H-Tucker (green) and T-Train (yellow) for a) PIE, b) COIL-100 and c) Hand Gesture datasets. Starting from the left for all (a), (b) and (c), the first two compression rates correspond to 1-scale MS-HoSVD with $\tau= 0.7$ and $\tau= 0.75$ while the last two are obtained from 2-scale approximation with $\tau= 0.7$ and $\tau= 0.75$, respectively. MS-HoSVD provides lower error than HoSVD, H-Tucker and T-Train.}
\label{data_reduction_bar}
\end{figure}

\begin{figure}[H]
\centering
\includegraphics[width=12cm]{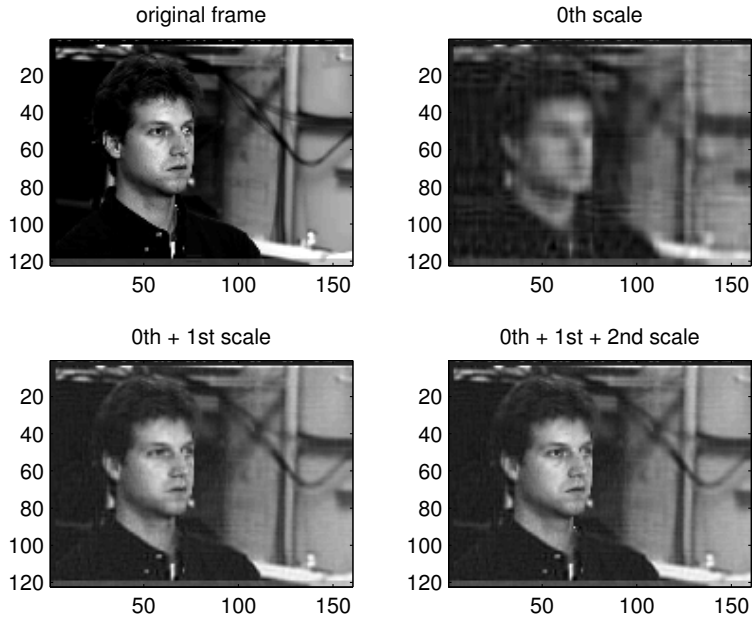}
\caption{A single frame of the PIE dataset showing increasing accuracy with scale for MS-HoSVD.}
\label{piefig}
\end{figure}

\subsection{Data Reduction with Adaptive Tree Pruning}

In this section, we evaluate the performance of adaptive tree pruning multiscale decompositions. In the pruning experiments, clustering is performed by LSA and the cluster number along each mode is chosen as $c_i=2$. The rank used in HoSVD is selected adaptively based on the energy threshold $\tau = 0.7$. A pruned version of 2-scale MS-HoSVD that greedily minimizes the  cost function $\mathbb{H}= Error + \lambda\cdot Compression$ for is implemented for PIE, COIL-100 and Hand Gesture datasets with varying $\lambda$ values as reported in Tables  \ref{PIEpruning}, \ref{COILpruning} and \ref{HGpruning}.  As $\lambda$ increases, reducing the compression rate becomes more important and the algorithm prunes the leaf nodes more.  For example, a choice of  $\lambda = 0.75$ prunes all of the nodes corresponding to the  second scale subtensors for PIE data (see Table \ref{PIEpruning}).

As can be seen from Tables \ref{PIEpruning}, \ref{COILpruning}, and \ref{HGpruning}, the best tradeoffs achieved between reconstruction error and compression rate occur at different $\lambda$ values for different datasets. For example, for PIE data, increasing $\lambda$ value does not provide much change in reconstruction error while increasing the compression. On the other hand, for COIL-100, $\lambda=0.75$ provides a good tradeoff between reconstruction error and compression rate. Small changes in $\lambda$ yield significant effects on pruning the subtensors of 2-scale decomposition of hand gesture data. 
Fig. \ref{piefigprune} illustrates the performance of the pruning algorithm on the PIE dataset. Applying pruning with $\lambda=0.25$ increases the reconstruction error from 0.0276 to 0.0506 while reducing the compression rate by a factor of 4 (Table \ref{PIEpruning}). As seen in Fig. \ref{piefigprune}, the $2^{\rm nd}$-scale approximation obtained by the adaptive pruning algorithm preserves most of the facial details in the image.

\begin{table}[h]
\footnotesize
\centering
\caption{Reconstruction error and compression rate computed for pruned tree structure obtained by applying MS-HoSVD with 2 scales to PIE data.}

\begin{tabular}{c||ccccccc}

 \hline
   $\lambda$     &0  & 0.22 & 0.25 &0.30  &0.75    \\ 

\hline

Normalized error   &0.0276 & 0.0395 &0.0506   & 0.0530 & 0.0540   \\
Compression   & 0.1241 & 0.0809 & 0.0377  & 0.0284  & 0.0261   \\
Scales of  subtensors & 0+1+2   & 0+1+2 & 0+1+2 &0+1+2  & 0+1   \\
\hline
\end{tabular}
\label{PIEpruning}
\end{table}

\begin{table}[h]
\footnotesize
\centering
\caption{Reconstruction error and compression rate computed for pruned tree structure obtained by applying MS-HoSVD with 2 scales to COIL-100 dataset.}

\begin{tabular}{c||cccccc}

 \hline
   $\lambda$    & 0  &  0.25 & 0.50 & 0.75 &0.80  \\ 
\hline
        
Normalized error   &  0.0857  &0.0867  & 0.0913 & 0.1060    &  0.1207     \\
Compression   & 0.0863  &0.0840   & 0.0734 & 0.0526  &  0.0347    \\        
Scales of subtensors & 0+1+2 & 0+1+2 & 0+1+2  &0+1+2  & 0+1+2   \\
\hline
\end{tabular}
\label{COILpruning}
\end{table}

\begin{table}[h]
\footnotesize
\centering
\caption{Reconstruction error and compression rate computed for pruned tree structure obtained by applying MS-HoSVD with 2 scales to Hand Gesture dataset.}
\begin{tabular}{c||cccccc}

 \hline
   $\lambda$    & 0  &  0.25 & 0.26 & 0.27 &0.28  \\ 
\hline
        
Normalized error  &  0.0691  &0.0869  & 0.0913 & 0.0946    &  0.0999     \\
Compression   & 0.1694  &0.1056   & 0.0827 & 0.0698  &  0.0514    \\        
Scales of subtensors & 0+1+2 & 0+1+2 & 0+1+2  &0+1+2  & 0+1   \\
\hline
\end{tabular}
\label{HGpruning}
\end{table}

\begin{figure}[H]
\centering
\includegraphics[width=12cm]{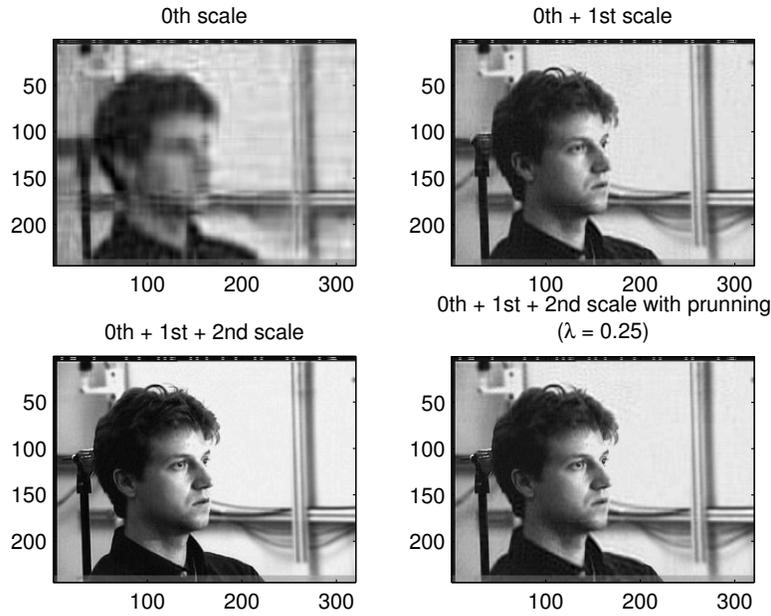}
\caption{Reconstruction error and compression rate computed for the pruned tree structure obtained by applying MS-HoSVD with 2 scales to the PIE dataset. The top-left and right images are sample frames obtained by reconstructing the tensor using only the $0^{\rm th}$ scale, and $0^{\rm th}$ and $1^{\rm st}$ scales, respectively.  The bottom-left image is a sample frame reconstructed using the $2$-scale approximation with all the sub-tensors, and the bottom-right image is the reconstruction using the $2$-scale analysis with the pruning approach where $\lambda=0.25$.}
\label{piefigprune}
\end{figure}

Performance of the pruning algorithm reported in Tables \ref{PIEpruning}, \ref{COILpruning} and \ref{HGpruning} is also compared with HoSVD, H-Tucker and T-Train decompositions in Fig. \ref{data_reduction_bar_ap}. As seen in Fig.  \ref{data_reduction_bar_ap} (b) and (c),  MS-HoSVD outperforms other approaches for compressing COIL-100 and Hand Gesture datasets at varying compression rates. However, for PIE data, the performance of MS-HoSVD and HoSVD are very close to each other while both approaches outperform H-Tucker and T-Train, as can be seen in Fig.  \ref{data_reduction_bar_ap} (a).
In Fig. \ref{piefigprune_comparison}, sample  frames of PIE data reconstructed by T-Train (top-left), H-Tucker (top-right), HoSVD (bottom-left) and pruned MS-HoSVD with 2-scales (bottom-right) are shown. It can be easily seen that the reconstructed images by H-Tucker and T-Train are more blurred than the ones obtained by HoSVD and MS-HoSVD. One can also see the facial details captured by MS-HoSVD are clearer than HoSVD although the performances of both algorithms are very similar to each other. The reason for capturing facial details better by MS-HoSVD is that the higher-scale subtensors encode facial details.

\begin{figure}[H]
\footnotesize{\centering}
\begin{minipage}[b]{0.48\linewidth}
\subfloat{\includegraphics[width=8 cm]{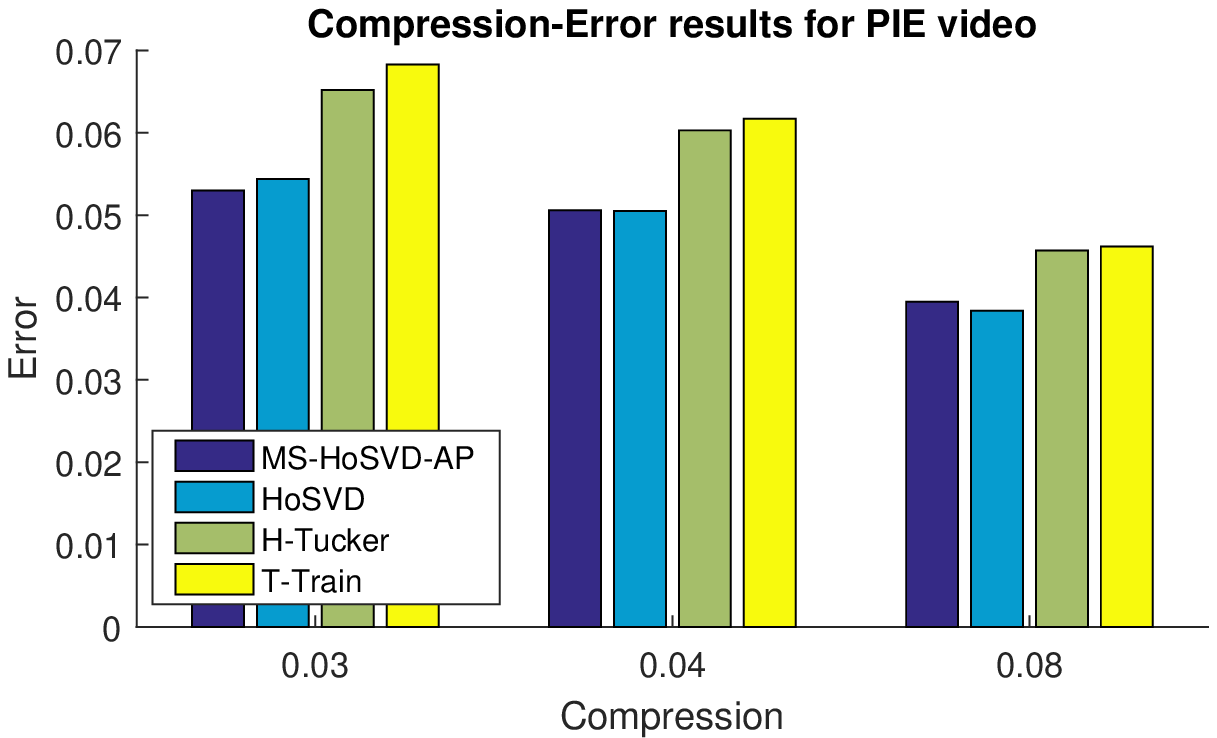}}\\
\centering (a) PIE \\
\end{minipage}
\hfill
\begin{minipage}[b]{0.48\linewidth}
\subfloat{\includegraphics[width=8 cm]{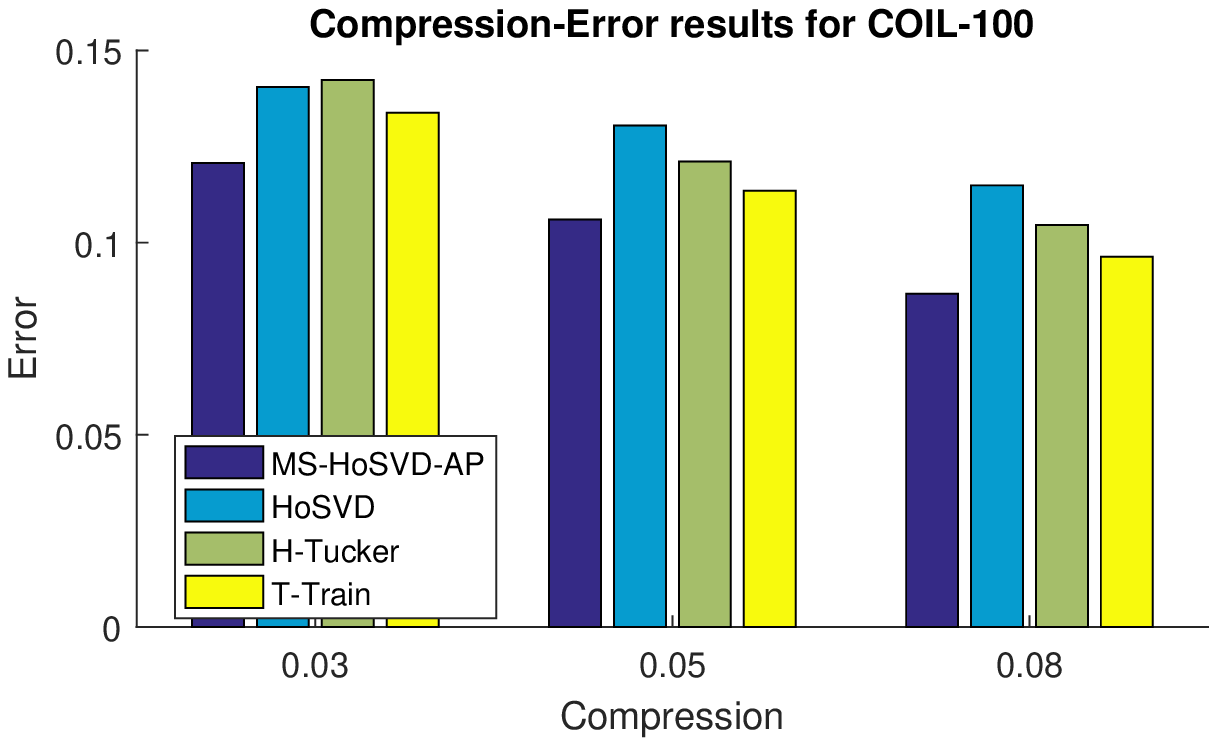}}\\
\centering (b) COIL-100\\
\end{minipage}
\centering
\begin{minipage}[b]{0.48\linewidth}
\subfloat{\includegraphics[width=8 cm]{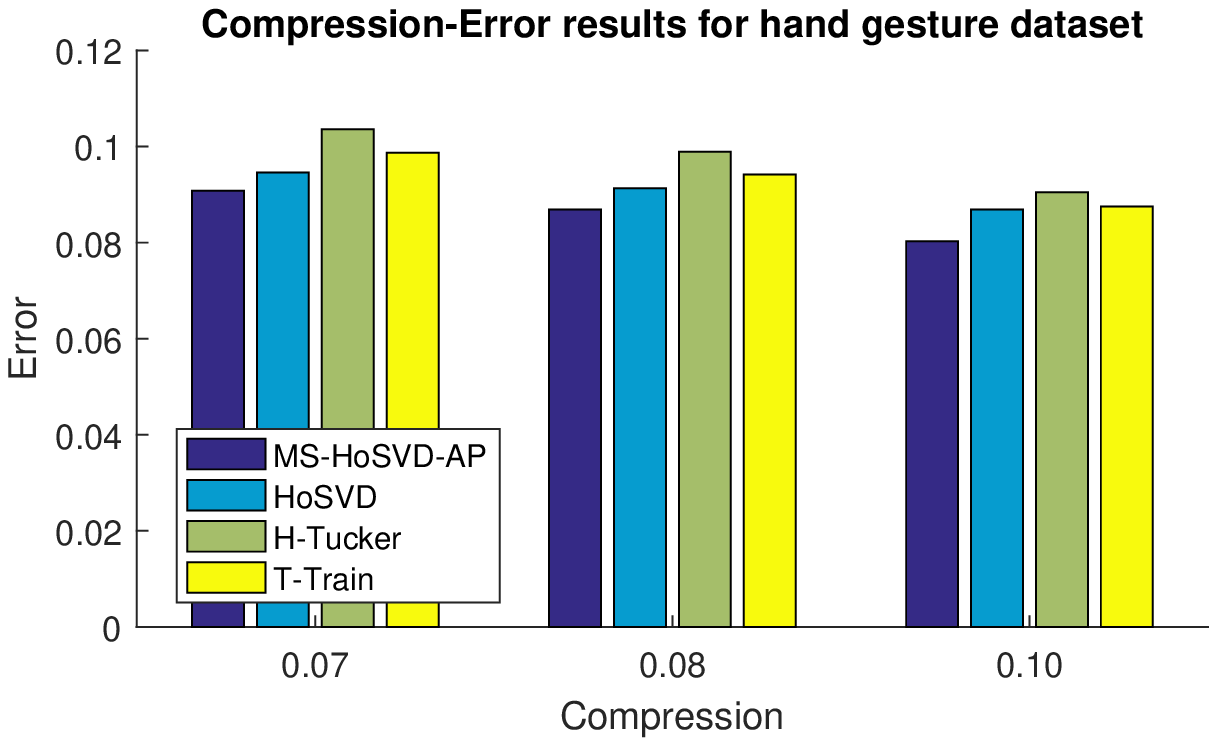}}\\
\centering (c) Hand Gesture\\
\end{minipage}


\caption{Compression rate versus Normalized Reconstruction Error for MS-HoSVD with adaptive pruning (dark blue), HoSVD (light blue), H-Tucker (green) and T-Train (yellow) for a) PIE, b) COIL-100 and c) Hand Gesture datasets. 2-scale MS-HoSVD tensor approximations are obtained using $\tau= 0.7$ for each scale and varying pruning trade-off parameter $\lambda$.}
\label{data_reduction_bar_ap}
\end{figure}

\begin{figure}[H]
\centering
\includegraphics[width=12cm]{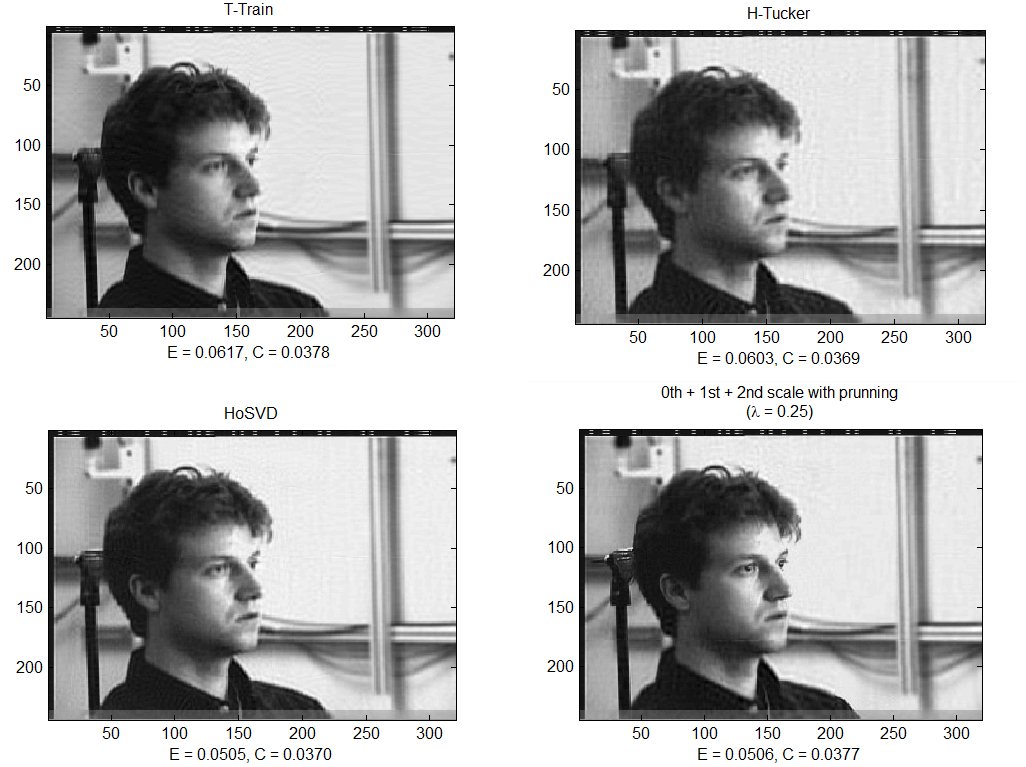}
\caption{Reconstructed frame samples from PIE data compressed by T-Train (top-left), H-Tucker (top-right), HoSVD
(bottom-left) and pruned MS-HoSVD with 2-scales (bottom-right). 2-scale MS-HoSVD tensor approximation is obtained using $\tau= 0.7$ for each scale and $\lambda=0.25$.}
\label{piefigprune_comparison}
\end{figure}

\section{Feature Extraction and Classification}
\label{sec:Classification}

In this section, we evaluate the features extracted from MS-HoSVD for classification of 2-mode and 3-mode tensors containing object images and hand gesture videos. 

The classification accuracy of MS-HoSVD features are compared to the features extracted by HoSVD and T-Train using three different classifiers: 1-NN, Adaboost and Naive Bayes.

\subsection{COIL-100 Image Dataset}
For computational efficiency, each image was downsampled to a gray-scale image of $32 \times 32$ pixels.  Number of images per object used for training data was gradually increased from 18 to 54 and selected randomly. A 3-mode tensor $\mathcal{X}^{tr}\in \mathbb{R}^{32\times 32\times I_{tr}}$ is constructed from training images  where $I_{tr} \in 100\times \left\lbrace  18,\; 36, \; 54 \right\rbrace$ and the rest of the images are used to create the testing tensor $\mathcal{X}^{te}\in \mathbb{R}^{32\times 32\times I_{te}}$  where $I_{te}=7200-I_{tr}$. 

\subsection{The Cambridge Hand Gesture Dataset}


For computational efficiency, each image  was downsampled to a gray-scale image   of $30 \times 40$ pixels. Number of image sequences used for training data gradually increased from 25 to 75 per gesture and selected randomly. A 4-mode tensor $\mathcal{X}^{tr}\in \mathbb{R}^{30\times 40\times  30 \times I_{tr}}$ is constructed from training image sequences  where $I_{tr} \in 9\times \left\lbrace 25,\; 50,\; 75 \right\rbrace$ and the rest of the image sequences are used to create the testing tensor $\mathcal{X}^{te}\in \mathbb{R}^{30\times 40\times  15\times I_{te}}$  where $I_{te}=900-I_{tr}$.
\subsection{Classification Experiments} 
\subsubsection{Training} 

For MS-HoSVD, the training tensor $\mathcal{X}^{tr}$ is decomposed using 1-scale MS-HoSVD as follows. Tensor partitioning is performed by LSA  and the cluster number along each mode is chosen as $c={\left\lbrace 2,\;3,\; 1\right\rbrace}$ yielding 6 subtensors for COIL-100 dataset and $c={\left\lbrace 2,\;2,\;3,\; 1\right\rbrace}$ yielding 12 subtensors for hand gesture dataset. We did not partition the tensor along the last mode that corresponds to the classes to make the comparison with other methods fair. The rank used in 0th scale is selected based on the energy criterion with $\tau=0.7$, while the full rank decomposition is used for the 1st scale. The 0th scale approximation
\begin{equation}
\hat{\mathcal{X}}_0^{tr}= \mathcal{C}_0^{tr}\times_1 \hat{\bf U}^{{tr},(1)} \times_2 \hat{\bf U}^{{tr},(2)}...\times_N \hat{\bf U}^{{tr},(N)}
\end{equation}
\noindent provides the $0^{\rm th}$-scale core tensor $\mathcal{C}_0^{tr}$, factor matrices $\hat{\bf U}^{tr,(i)}$ and residual tensor $\mathcal{W}_0^{tr}= \mathcal{X}^{tr}- \hat{\mathcal{X}}_0^{tr}$. Next, the $0^{\rm th}$-scale feature tensor $\mathcal{S}^{tr}_0$ for the training data is created by projecting ${\mathcal{X}}^{tr}$s onto the first $N-1$ factor matrices ${\bf U}^{tr,(i)}$ as:
\begin{equation}
\mathcal{S}^{tr}_0= {\mathcal{X}}^{tr} \times_1 \left(\hat{\bf U}^{tr,(1)}\right)^{\top} \times_2 \left(\hat{\bf U}^{tr,(2)}\right)^{\top}...\times_{N-1} \left(\hat{\bf U}^{tr,(N-1)}\right)^{\top}.
\end{equation}


Subtensors of 
$\mathcal{W}_0^{tr}$ obtained by $\mathcal{X}^{tr}_{1,k}= \mathcal{W}_0^{tr}(J^{tr,1}_{1,k}\times J^{tr,2}_{1,k} \times ...\times J^{tr,N}_{1,k})$
are used to extract $1^{st}$-order core tensors $\mathcal{C}_{1,k}$ and factor matrices ${\bf U}^{{tr},(i)}_{1,k}$ as: 
\begin{equation}
{\mathcal{X}}_{1,k}^{tr} = \mathcal{C}_{1,k}^{tr}\times_1 {\bf U}^{{tr},(1)}_{1,k} \times_2 {\bf U}^{{tr}(2)}_{1,k}...\times_N {\bf U}^{{tr}(N)}_{1,k}.
\end{equation}

$1st$-order feature tensors are then created by projecting ${\mathcal{X}}^{tr}_{1,k}$s onto the first $N-1$ factor matrices ${\bf U}^{tr,(i)}_{1,k}$ as:
\begin{equation}
\mathcal{S}^{tr}_{1,k}= {\mathcal{X}}^{tr}_{1,k} \times_1 \left({\bf U}^{tr,(1)}_{1,k}\right)^{\top} \times_2 \left({\bf U}^{tr,(2)}_{1,k}\right)^{\top}...\times_{N-1} \left({\bf U}^{tr,(N-1)}_{1,k}\right)^{\top}.
\end{equation}

Unfolding the feature tensors $\mathcal{S}^{tr}_0$ and $\mathcal{S}^{tr}_{1,k}$ along the sample mode $N$  and concatenating them to each other yields a high dimensional feature vector for each of the training samples. From these vectors, $N_f$ features with the highest  Fisher Score \cite{gu2012generalized} are selected to form the lower-dimensional feature vectors ${\bf x}^{tr}\in \mathbb{R}^{N_f\times 1}$ for each training sample where the number of features ($N_f$) is determined 100 for COIL-100 and 200 for hand gesture dataset emprically. For HoSVD and T-Train, full rank decompositions are computed and feature vectors are created by selecting $N_f$  features with the highest Fisher Score from the core tensors as described above. For T-Train, the procedure described in \cite{bengua2016matrix} is used without reducing the dimensionality.

\subsubsection{Testing}

To create the $0^{\rm th}$-order feature tensor $\mathcal{S}^{te}_0$ for testing samples, first, the testing tensor $\mathcal{X}_{te}$ is projected onto  $\hat{\bf U}^{tr,(i)}$ where $i\in \left[N-1\right]$ as:
\begin{equation}
\mathcal{S}^{te}_0= {\mathcal{X}}^{te} \times_1 \left(\hat{\bf U}^{tr,(1)}\right)^{\top} \times_2 \left(\hat{\bf U}^{tr,(2)}\right)^{\top}...\times_{N-1} \left(\hat{\bf U}^{tr,(N-1)}\right)^{\top}.
\end{equation}

The $0^{\rm th}$-order residual tensor $\mathcal{W}_0^{te}$ of testing data is computed as $\mathcal{W}_0^{te}={\mathcal{X}}^{te}- {\mathcal{X}}^{te}\bigtimes_{n=1}^{n=N}\left(\hat{\bf U}^{tr,(n)}\right)^{\top}$. Then $1^{\rm st}$-order subtensors are created from $\mathcal{W}_0^{te}$ using the same partitioning as the 0th order training residual tensor $\mathcal{W}_0^{tr}$ as $\mathcal{X}^{te}_{1,k}= \mathcal{W}_0^{te}(J^{tr,1}_{1,k}\times J^{tr,2}_{1,k} \times ...\times J^{tr,N}_{1,k})$. The $1^{\rm st}$-order feature tensors $\mathcal{S}^{te}_{1,k}$ for the testing samples are then obtained by  
\begin{equation}
\mathcal{S}^{te}_{1,k}= {\mathcal{X}}^{te}_{1,k} \times_1 \left({\bf U}^{tr,(1)}_{1,k}\right)^{\top} \times_2 \left({\bf U}^{tr,(2)}_{1,k}\right)^{\top}...\times_{N-1} \left({\bf U}^{tr,(N-1)}_{1,k}\right)^{\top}.
\end{equation}

Similar to the training step, unfolding the feature tensors $\mathcal{S}^{te}_0$ and $\mathcal{S}^{te}_{1,k}$ along the sample mode $N$  and concatenating them with each other yields high dimensional feature vectors for the testing samples. The  features  corresponding to the features selected from the training step are used to form the feature vectors for testing samples ${\bf x}^{te}\in \mathbb{R}^{N_f\times 1}$. 
A similar two-step procedure, i.e projecting the testing tensor onto training factor matrices  followed by selecting $N_f$ features, is used to create testing feature vectors  for HoSVD and T-Train. Discrimination performance of the  feature vectors are evaluated using different classifiers including 1-NN, Adaboost and Naive Bayes. 
 

Tables \ref{objectClassification} and \ref{HGClassification} summarize the classification accuracy for the three methods using three different classifiers for COIL-100 and Hand gesture data sets, respectively. As it can be seen from these Tables, for both data sets and all classifiers MS-HoSVD performs the best except for a Naive Bayes Classifier trained by $25\%$ of the data to classify hand gesture dataset. As seen in Tables \ref{objectClassification} and \ref{HGClassification}, the performance of HoSVD, T-Train and MS-HoSVD become close to each other as  the size of the training dataset increases, as expected. The reason for the superior performance of MS-HoSVD is that MS-HoSVD captures the variations and nonlinearities across the modes such as rotation or translation better than the other methods. In both of the datasets used in this section, the images are rotated across different frames. Since these nonlinearities are encoded in the higher-scale ($1^{\rm st}$-scale) features  while the average characteristics, which are the same as HoSVD, are captured by the lower scale ($0^{\rm th}$-scale) MS-HoSVD features, the classification performance of the MS-HoSVD is slightly better than HoSVD. It is also seen that T-Train features are not as good as MS-HoSVD and HoSVD features for capturing rotations and translations in the data and requires larger training set to reach the performance of MS-HoSVD and HoSVD.



\begin{table}[H]
\footnotesize
\centering
\caption{Classification results for COIL-100 dataset over 20 trials with $N_f=100$.}

\begin{tabular}{c|c||cccccccccc|}

 \hline
\multirow{ 2}{*}{Training Size} & \multirow{ 2}{*}{Method} & \multicolumn{2}{|c}{1-NN}& \multicolumn{2}{|c}{Adaboost} &\multicolumn{2}{|c}{Bayes}\\ 
&  & \multicolumn{2}{|c}{mean $\pm $ std.}& \multicolumn{2}{|c}{mean $\pm $ std.}& \multicolumn{2}{|c}{mean $\pm $ std.}\\
\hline

\multirow{ 3}{*}{25$\%$}  

& MS-HoSVD & \multicolumn{2}{|c}{\bf 93.71 $\pm $ 1.28}& \multicolumn{2}{|c}{\bf 88.90 $\pm $ 1.24 } &\multicolumn{2}{|c}{\bf 89.88 $\pm $ 2.08 }\\

& HoSVD & \multicolumn{2}{|c}{93.07 $\pm $ 1.33}& \multicolumn{2}{|c}{87.47 $\pm $ 1.53 }&\multicolumn{2}{|c}{ 87.36 $\pm $ 3.32 }\\
 
 & T-Train & \multicolumn{2}{|c}{92.29 $\pm $ 2.23}& \multicolumn{2}{|c}{ 87.26 $\pm $ 1.84 }&\multicolumn{2}{|c}{ 88.43 $\pm $ 1.36 }\\

 \hline
\multirow{ 3}{*}{50$\%$}  

& MS-HoSVD & \multicolumn{2}{|c}{\bf 97.41 $\pm $ 0.69}& \multicolumn{2}{|c}{ \bf 92.54 $\pm $ 1.58}&\multicolumn{2}{|c}{\bf 91.62 $\pm $ 1.21 }\\

& HoSVD & \multicolumn{2}{|c}{97.10 $\pm $ 0.83}& \multicolumn{2}{|c}{ 91.29 $\pm $ 2.16 }&\multicolumn{2}{|c}{ 90.43 $\pm $ 1.60 }\\

 & T-Train & \multicolumn{2}{|c}{96.99 $\pm $ 1.09}& \multicolumn{2}{|c}{ 91.36 $\pm $ 1.44 }&\multicolumn{2}{|c}{  91.21 $\pm $ 1.66 }\\

 \hline
\multirow{ 3}{*}{ 75$\%$}  

& MS-HoSVD & \multicolumn{2}{|c}{ \bf 98.38 $\pm $ 0.65}& \multicolumn{2}{|c}{\bf 93.66 $\pm $1.35 }&\multicolumn{2}{|c}{\bf 92.60 $\pm $ 1.79 }\\

& HoSVD & \multicolumn{2}{|c}{98.16 $\pm $ 0.72}& \multicolumn{2}{|c}{ 92.23 $\pm $1.46  }&\multicolumn{2}{|c}{ 92.38 $\pm $ 1.92 }\\

 & T-Train & \multicolumn{2}{|c}{98.25 $\pm $ 0.41}& \multicolumn{2}{|c}{ 93.10 $\pm $1.33 }&\multicolumn{2}{|c}{ 92.33 $\pm $ 1.52 }\\

\hline

\end{tabular}
\label{objectClassification}
\end{table}

\begin{table}[H]
\footnotesize
\centering
\caption{Classification results for hand gesture dataset over 20 trials with $N_f=200$.}

\begin{tabular}{c|c||cccccccc|}

 \hline
\multirow{ 2}{*}{Training Size} & \multirow{ 2}{*}{Method} & \multicolumn{2}{|c}{1-NN}& \multicolumn{2}{|c}{ Adaboost}&\multicolumn{2}{|c}{Bayes}\\  
&  & \multicolumn{2}{|c}{mean $\pm $ std.}& \multicolumn{2}{|c}{mean $\pm $ std.}& \multicolumn{2}{|c}{mean $\pm $ std.}\\
\hline

\multirow{ 3}{*}{25$\%$}  
& MS-HoSVD & \multicolumn{2}{|c}{ \bf 75.40 $\pm $ 3.87}& \multicolumn{2}{|c}{\bf 77.81 $\pm $ 1.87}& \multicolumn{2}{|c}{82.89 $\pm $ 1.84}\\

& HoSVD & \multicolumn{2}{|c}{75.01 $\pm $ 3.99}& \multicolumn{2}{|c}{77.37 $\pm $ 2.51}& \multicolumn{2}{|c}{\bf 83.11 $\pm $ 2.89}\\
 & T-Train & \multicolumn{2}{|c}{69.20 $\pm $ 2.63}& \multicolumn{2}{|c}{67.53 $\pm $ 2.27}& \multicolumn{2}{|c}{ 72.99 $\pm $ 4.91}\\

 \hline
\multirow{ 3}{*}{50$\%$}  
& MS-HoSVD & \multicolumn{2}{|c}{\bf 83.86 $\pm $ 3.12}& \multicolumn{2}{|c}{\bf 85.46 $\pm $ 1.66}&  \multicolumn{2}{|c}{\bf 85.86 $\pm $ 2.25}\\

& HoSVD & \multicolumn{2}{|c}{83.15 $\pm $ 2.90}& \multicolumn{2}{|c}{85.14 $\pm $ 1.45}& \multicolumn{2}{|c}{84.00 $\pm $ 1.68}\\
 & T-Train & \multicolumn{2}{|c}{78.97 $\pm $ 2.25}& \multicolumn{2}{|c}{ 78.82 $\pm $ 2.87}& \multicolumn{2}{|c}{ 81.46 $\pm $ 1.37}\\

 \hline
\multirow{ 3}{*}{ 75$\%$} 
 & MS-HoSVD & \multicolumn{2}{|c}{\bf 87.47 $\pm $ 2.07}& \multicolumn{2}{|c}{\bf 88.15 $\pm $ 2.20}& \multicolumn{2}{|c}{\bf 86.75 $\pm $ 2.63}\\

& HoSVD & \multicolumn{2}{|c}{86.93 $\pm $ 2.31}& \multicolumn{2}{|c}{88.08 $\pm $ 2.51}& \multicolumn{2}{|c}{85.04 $\pm $ 2.41}\\
 & T-Train & \multicolumn{2}{|c}{85.64 $\pm $ 2.57}& \multicolumn{2}{|c}{ 82.60 $\pm $ 3.16}& \multicolumn{2}{|c}{83.64 $\pm $ 2.31}\\

\hline

\end{tabular}
\label{HGClassification}
\end{table}

\section{Conclusions}

In this paper, we proposed a new multi-scale tensor decomposition technique for better approximating the local nonlinearities in generic tensor data. The proposed approach constructs a tree structure by considering similarities along different fibers of the tensor and decomposes the tensor into lower dimensional subtensors hierarchically. A low-rank approximation of each subtensor is then obtained by HoSVD. We also introduced a pruning strategy to find the optimum tree structure by keeping the important nodes and  eliminating redundancy in the data. The proposed approach is applied to a set of $3$-way and $4$-way tensors to evaluate its performance on both data reduction and classification applications. As it is illustrated in sections \ref{sec:DataCompress} and \ref{sec:Classification}, any application involving tensor data reduction and classification would benefit from the proposed method. Some examples include hyper-spectral image compression, high-dimensional video clustering and functional connectivity network analysis in neuroscience.

Although this paper focused on the integration of a single existing tensor factorization technique (i.e., the HoSVD) into a clustering-enhanced multiscale approximation framework, we would like to emphasize that the ideas presented herein are significantly more general.  In principal, for example, there is nothing impeding the development of multiscale variants of other tensor factorization approaches (e.g., PARAFAC, T-Train, H-Tucker, etc.) in essentially the same way.  In this paper it is demonstrated that the use of the HoSVD as part of a multiscale approximation approach leads to improved compression and classification performance over standard HoSVD approaches.  However, this paper should additionally be considered as evidence that similar improvements are also likely possible for other tensor factorization-based compression and classification schemes, as well as for other related applications.


Future work will consider automatic selection of parameters such as the number of clusters  and the appropriate rank along each mode. The computational efficiency of the proposed method can also be improved through parallelization of the algorithm by, e.g., constructing the disjoint subtensors at each scale in parallel, as well as by utilizing distributed and parallel SVD algorithms such as \cite{iwen2016distributed} when computing their required HoSVD decompositions (see also, e.g., \cite{liavas2015parallel} for other related parallel implementations). Such efficient implementations will enable the computation of finer-scale decompositions for higher-order and higher-dimensional tensors.

\bibliographystyle{IEEEtran}
\bibliography{ConsensusReferences,trackingRefs,tenDec}
\appendix

\subsection{Effective Partitioning, and Error Analysis}
\label{sec:AppendixErrorBound}

In order to facilitate error analysis for the 1-scale MS-HoSVD that is similar to the types of error analysis available for various HoSVD-based low-rank approximation strategies (see, e.g., \cite{vannieuwenhoven2012new}), we will engage in a more in depth discussion of condition \eqref{equ:ExampleErrorCond} herein.  Recall that the partition of $\mathcal{W}_0$ formed by the restriction matrices ${\bf R}_{k}^{(n)}$ in \eqref{equ:RestrictionMat} -- \eqref{equ:SubTensorAddtoBigTensor} is called \textit{effective} if there exists another \textit{pessimistic} partitioning of $\mathcal{W}_0$ via restriction matrices
$\left\lbrace \tilde{\bf R}_{k}^{(n)}\right\rbrace^K_{k=1}$ together with a bijection $f:[K]\rightarrow[K]$ such that
\begin{equation}
\sum_{n=1}^N \left\| {\mathcal{X}}|_k\times_n \left({\bf I} - {\bf Q}^{(n)}_k \right) \right\|^2\leq \sum_{n=1}^N \left\| \mathcal{W}_0 \times_n \tilde{\bf R}_{f(k)}^{(n)} \left({\bf I}-\tilde{\bf P}^{(n)} \right)\bigtimes_{h\neq n}^N\tilde{\bf R}_{f(k)}^{(h)}\right\|^2
\label{eqn_definition}
\end{equation}
\noindent  holds for each $k\in [K]$. In \eqref{eqn_definition} the $\left\lbrace \tilde{\bf P}^{(n)}\right\rbrace$ are the orthogonal projection matrices obtained from the HoSVD of $\mathcal{W}_0$ with ranks $\tilde{r}_n \geq \bar{r}_n \geq r_n$ \big(i.e., where each $\tilde{{\bf P}}^{(n)}$ projects onto the top $\tilde{r}_n$ left singular vectors of the matricization ${\bf W}_{0,(n)}$\big).  
Below we will show that \eqref{eqn_definition} holding for $\mathcal{W}_0$ implies that the error $\| \mathcal{W}_1 \|$ resulting from our $1^{\rm st}$-scale approximation in \eqref{equ:ErrorFormulaForScale1} is less than an upper bound of the type given for a high-rank standard HoSVD-based approximation \eqref{equ:BigHoSVDapprox} in \cite{vannieuwenhoven2012new}.

Considering condition \eqref{eqn_definition} above, we note that experiments show that it is regularly satisfied on real datasets when $(i)$ the effective restriction matrices $\left\lbrace {\bf R}_{k}^{(n)} \right\rbrace^K_{k=1}$ in \eqref{equ:RestrictionMat} -- \eqref{equ:SubTensorAddtoBigTensor} are first formed by clustering the rows of each unfolding of $\mathcal{W}_0$ using, e.g., local subspace analysis (LSA), after which $(ii)$ pessimistic restriction matrices $\left\lbrace \tilde{\bf R}_{k}^{(n)}\right\rbrace^K_{k=1}$ are randomly generated in order to create another (random) partition of $\mathcal{W}_0$ into $K$ different disjoint subtensors for comparison.  The bijection $f$ can then be created by, e.g., $(i)$ sorting the left-hand side errors in \eqref{eqn_definition} for each $k \in [k]$, $(ii)$ sorting the right-hand side errors in \eqref{eqn_definition} for each $k \in [K]$, and then $(iii)$ matching the largest left-hand and right-hand errors for comparision, the second largest left-hand and right-hand errors for comparision, etc..  When checked in this way the sorted right-hand side errors often dominate (entrywise) the sorted left-hand side errors for various reasonable ranks $\bar{r}_n = \tilde{r}_n = r_H = r_0 + c^{N-1}r_1$ (as a function of $r_0$ and $r_1$ with, e.g., $c = 2$) on every dataset considered in Section~\ref{sec:DataCompress} above, thereby verifying that \eqref{eqn_definition} does indeed regularly hold.


We will now begin to prove Theorem~\ref{MainThm} with a lemma that shows our subtensor-based approximation of $\mathcal{W}_0$ is accurate whenever \eqref{eqn_definition} is satisfied.

\begin{lem}
	Let $\mathcal{W}_0 = \mathcal{X} - \hat{\mathcal{X}}_0 \in \mathbb{R}^ {I_1 \times I_2 \times ... I_N }$.  Suppose that $\left\lbrace {\bf R}_{k}^{(n)}\right\rbrace$ is a collection of effective restriction matrices that form an effective partition of $\mathcal{W}_0$ with respect to a pessimistic partition formed via pessimistic restriction matrices $\left\lbrace \tilde{\bf R}_k^{(n)} \right\rbrace$ as per \eqref{eqn_definition} above.
	Similarly, let $\tilde{\bf P}^{(n)}$ be the rank $\tilde{r}_n \geq \bar{r}_n~\forall n$ orthogonal projection matrices from \eqref{eqn_definition} obtained via the truncated HoSVD of $\mathcal{W}_0$ as above. Then,
	\begin{equation*}
	\left\| \mathcal{W}_0 -  \hat{\mathcal{X}}_1 \right\|^2 = \left\| \left( \mathcal{X} - \hat{\mathcal{X}}_0 \right) - \sum^K_{k = 1} \left( \left( \mathcal{X} - \hat{\mathcal{X}}_0 \right) \bigtimes_{n=1}^N {\bf Q}^{(n)}_k \right) \right\|^2\leq \sum_{n=1}^N \left\| \left( \mathcal{X} - \hat{\mathcal{X}}_0 \right) \times_n \left( {\bf I}-\tilde{\bf P}^{(n)} \right)\right\|^2.
	\end{equation*}
	\label{lem:PartitionW0Approx}
\end{lem}

\begin{proof}
	We have that
	\begin{align}
	\left\| \mathcal{W}_0 -  \hat{\mathcal{X}}_1 \right\|^2 &= \left\| \mathcal{W}_0 -  \sum^K_{k = 1} \mathcal{W}_0 \bigtimes_{n=1}^N {\bf Q}^{(n)}_k \right\|^2  &\left(\textrm{Using \eqref{equ:ResidualDef} and \eqref{equ:Scale1ApproxDef}} \right) \nonumber\\
	&=
	\left\| \sum_{k=1}^K \mathcal{W}_0 \bigtimes_{n=1}^N {\bf R}_{k}^{(n)} -  \sum_{k=1}^K \mathcal{W}_0 \bigtimes_{n=1}^N {\bf Q}^{(n)}_k {\bf R}_{k}^{(n)} \right\|^2 & \left(\textrm{Using \eqref{equ:subTensorDef}, \eqref{equ:SubTensorAddtoBigTensor}, and \eqref{equ:PropResProjStuff}} \right) \nonumber\\
	&= 
	\left\|\sum_{k=1}^K \mathcal{W}_0 \bigtimes_{n=1}^N \left( {\bf R}_{k}^{(n)} - {\bf Q}^{(n)}_k {\bf R}_{k}^{(n)}\right)\right\|^2 &\left(\textrm{Using Lemma~\ref{lem:modeProdProps}} \right) \nonumber\\
	&= 
	\sum_{k=1}^K \left\| \mathcal{X}|_k \bigtimes_{n=1}^N \left( {\bf I} - {\bf Q}^{(n)}_k \right) \right\|^2. & \left(\textrm{Using Lemma~\ref{lem:modeProdProps}, \eqref{equ:subTensorDef}, \eqref{equ:PropResProjStuff}, and support disjointness} \right)
	\label{eqnL1_1}
	\end{align}  
	
	Applying lemmas~\ref{lem:modeProdProps}~and~\ref{lem:TensorPythagorean} to \eqref{eqnL1_1} we can now see that
	\begin{align*}
	\left\| \mathcal{W}_0 -  \hat{\mathcal{X}}_1 \right\|^2 =& \sum_{k=1}^K \sum_{n=1}^N \left\| \mathcal{X}|_k \bigtimes_{h=1}^{n-1} {\bf Q}^{(h)}_k \times_{n} \left({\bf I} - {\bf Q}^{(n)}_k \right) \right\|^2 
	\nonumber \\
	\leq& \sum_{k=1}^K \sum_{n=1}^N \left\| {\mathcal{X}}|_k\times_n \left({\bf I} - {\bf Q}^{(n)}_k \right) \right\|^2
	\end{align*}
	since the ${\bf Q}^{(n)}_k$ matrices are orthogonal projections.  Using assumption \eqref{eqn_definition} we now get that
	\begin{align*}
	\left\| \mathcal{W}_0 -  \hat{\mathcal{X}}_1 \right\|^2 \leq& \sum_{k=1}^K \sum_{n=1}^N \left\| \mathcal{W}_0 \times_n \tilde{\bf R}_{k}^{(n)} \left({\bf I}-\tilde{\bf P}^{(n)} \right) \bigtimes_{h\neq n}^{N} \tilde{\bf R}_{k}^{(h)}
	\right\|^2
	\\
	=& \sum_{n=1}^N \left\| \mathcal{W}_0 \times_n \left({\bf I}-\tilde{\bf P}^{(n)} \right) \right\|^2
	\end{align*}
	where we have used the fact that the pessimistic restriction matrices $\tilde{\bf R}_k^{(n)}$ partition $\mathcal{W}_0$ in the last line.
\end{proof}

Lemma~\ref{lem:PartitionW0Approx} indicates that the error in approximating $\mathcal{W}_0$ via low-rank approximations of its effective subtensors is potentially smaller than the error obtained by approximating $\mathcal{W}_0$ via (higher-rank) truncated HoSVDs whenever \eqref{eqn_definition} holds.\footnote{That is, the upper bound on the error provided by Lemma~\ref{lem:PartitionW0Approx} is less than or equal to the upper bound on the error for truncated HoSVDs provided by, e.g., \cite{vannieuwenhoven2012new} when/if \eqref{eqn_definition} holds.}  The following theorem shows that this good error behavior extends to the entire $1^{\rm st}$-scale approximation provided by \eqref{equ:ErrorFormulaForScale1} whenever \eqref{eqn_definition} holds.

\begin{thm}[Restatement of Theorem~\ref{MainThm}]
	Let $\mathcal{X} \in \mathbb{R}^{I_1\times I_2...\times I_N}$.  Suppose that \eqref{eqn_definition} holds.  Then, the first-scale approximation error given by MS-HoSVD \eqref{equ:ErrorFormulaForScale1} is bounded by
	\begin{align*}
	\left\| \mathcal{W}_1 \right\|^2 = \left\| \mathcal{X}-\hat{\mathcal{X}}_0-\hat{\mathcal{X}}_1 \right\|^2 \leq & \sum_{n=1}^N \left\| \mathcal{X}\times_n \left( {\bf I}-\bar{\bf P}^{(n)} \right) \right\|^2
	\end{align*}
	where $\left\lbrace\bar{\bf P}^{(n)}\right\rbrace$ are low-rank  projection matrices of rank $\bar{r}_n \geq r_n$ obtained from the truncated HoSVD of $\mathcal{X}$ as per \eqref{equ:BigHoSVDapprox}.
\end{thm}

\begin{proof}
	Using \eqref{equ:ResidualDef} and \eqref{equ:FirstScaleResidualDef} together with lemma~\ref{lem:PartitionW0Approx} we can see that 
	\begin{align}
	\left\| \mathcal{W}_1 \right\|^2 = \left\| \mathcal{X}-\hat{\mathcal{X}}_0-\hat{\mathcal{X}}_1 \right\|^2 = \left\| \mathcal{W}_0 -\hat{\mathcal{X}}_1 \right\|^2 &\leq \sum_{n=1}^N \left\| \left( \mathcal{X} - \hat{\mathcal{X}}_0 \right) \times_n \left( {\bf I}-\tilde{\bf P}^{(n)} \right)\right\|^2 \nonumber\\
	&\leq \sum_{n=1}^N \left\| \left( \mathcal{X} - \hat{\mathcal{X}}_0 \right) \times_n \left({\bf I}-\bar{\bf Q}^{(n)} \right) \right\|^2
	\label{equ:BestRankrtildeApprox}
	\end{align}
	where $\bar{\bf Q}^{(n)} \in \mathbb{R}^{I_n \times I_n}$ is the orthogonal projection matrix of rank $\tilde{r}_n$ which projects onto the subspace spanned by the top $\tilde{r}_n$ left singular vectors of ${\bf X}_{(n)}$.  Here \eqref{equ:BestRankrtildeApprox} holds because the orthogonal projection matrices $\tilde{\bf P}^{(n)}$ are chosen in \eqref{eqn_definition} so that $\tilde{\bf P}^{(n)}{\bf W}_{0,(n)}$ is a best possible rank $\tilde{r}_n$ approximation to ${\bf W}_{0,(n)}$.  As a result, we have that 
	$$\left\| \left( \mathcal{X} - \hat{\mathcal{X}}_0 \right) \times_n \left( {\bf I}-\tilde{\bf P}^{(n)} \right)\right\|^2 = \left\| \left( {\bf I}-\tilde{\bf P}^{(n)} \right) {\bf W}_{0,(n)} \right\|^2_{\rm F} \leq \left\| \left( {\bf I}-\bar{\bf Q}^{(n)} \right) {\bf W}_{0,(n)} \right\|^2_{\rm F} = \left\| \left( \mathcal{X} - \hat{\mathcal{X}}_0 \right) \times_n \left({\bf I}-\bar{\bf Q}^{(n)} \right) \right\|^2$$
	must hold for each $n \in [N]$.
	
	Continuing from \eqref{equ:BestRankrtildeApprox} we can use the definition of $\hat{\mathcal{X}}_0$ in \eqref{equ:LittleHoSVDapprox} to see that 
	\begin{align}
	\left\| \mathcal{W}_1 \right\|^2 &\leq \sum_{n=1}^N \left\| \left( \mathcal{X} - \mathcal{X} \bigtimes_{h=1}^N {\bf P}^{(h)} \right) \times_n \left({\bf I}-\bar{\bf Q}^{(n)} \right) \right\|^2 \nonumber\\
	&= \sum_{n=1}^N \left\| \mathcal{X} \times_n \left({\bf I}-\bar{\bf Q}^{(n)} \right)- \mathcal{X} \bigtimes_{h=1}^N {\bf P}^{(h)} \times_n \left({\bf I}-\bar{\bf Q}^{(n)} \right) \right\|^2
	\label{equ:ThmAlmostdone}
	\end{align}
	by lemma~\ref{lem:modeProdProps}.  Due to the definition of $\bar{\bf Q}^{(n)}$ together with the fact that its rank is $\tilde{r}_n \geq r_n$ we can see that $\left({\bf I}-\bar{\bf Q}^{(n)} \right) {\bf P}^{(n)} = {\bf 0}$.  As a consequence, lemma~\ref{lem:modeProdProps} implies that $\mathcal{X} \bigtimes_{h=1}^N {\bf P}^{(h)} \times_n \left({\bf I}-\bar{\bf Q}^{(n)} \right) = {\bf 0}$ for all $n \in [N]$.  Continuing from \eqref{equ:ThmAlmostdone} we now have that 
	\begin{equation*}
	\left\| \mathcal{W}_1 \right\|^2 \leq \sum_{n=1}^N \left\| \mathcal{X} \times_n \left({\bf I}-\bar{\bf Q}^{(n)} \right) \right\|^2.
	\end{equation*}
	Again, appealing to the definition of both $\bar{\bf Q}^{(n)}$ and $\bar{\bf P}^{(n)}$ in \eqref{equ:BigHoSVDapprox}, combined with the fact that $\tilde{r}_n \geq \bar{r}_n$, finally yields the desired result.
\end{proof}

We refer the reader to the strong empirical performance of MS-HoSVD in Section~\ref{sec:DataCompress} for additional evidence supporting the utility of \eqref{equ:ErrorFormulaForScale1} as a means of improving the compression performance of standard HoSVD-based compression techniques.  In addition, we further refer the reader to Section~\ref{sec:Classification} where it is empirically demonstrated that MS-HoSVD is also capable of selecting more informative features than HoSVD-based methods for the purposes of classification.  These two facts together provide strong evidence that combining the use of clustering-enhanced multiscale approximation with existing tensor factorization techniques can lead to improved performance in multiple application domains.

\subsection{Experiment for Error Analysis}
\label{app:Example}

In this experiment we evaluate the error obtained by the $1^{\rm st}$-scale MS-HoSVD analysis of a tensor along the lines of the model described in Section \ref{multiScaleMethod}.
Herein we consider a three-way tensor $\mathcal{X}\in \mathbb{R}^{20\times20\times20}$ that is the sum of two tensors as $\mathcal{X}=\mathcal{X}_0 + \mathcal{X}_1$ where  $\mathcal{X}_0\in \mathbb{R}^{20\times20\times20}$ has $n$-rank $\left( 2,\;2,\;2 \right)$, and $\mathcal{X}_1\in \mathbb{R}^{20\times20\times20}$ is formed by concatenating $8$ subtensors $\mathcal{X}_k\in \mathbb{R}^{10\times10\times10}$ each also with $n$-rank $\left( 2,\;2,\;2 \right)$. Low-rank approximations for $\mathcal{X}$ and its subtensors are always obtained via the truncated HoSVD.  The 1-scale MS-HoSVD is applied with the ground truth partitions ${\bf R}_{k}^{(n)}$, partitions provided by Local Subspace Analysis (LSA) clustering $\hat{\bf R}_{k}^{(n)}$, and also with randomly chosen partitions $\tilde{\bf R}_{k}^{(n)}$ of the $0^{\rm th}$-scale residual error $\mathcal{W}_0$ into $8$ different $10\times10\times10$ subtensors. For the LSA clustering, the cluster numbers are selected as $2$ along each mode also yielding $8$ subtensors.

The $0^{\rm th}$-scale $n$-rank for MS-HoSVD is selected as $\left( 2,\;2,\;2 \right)$, and the $1^{\rm st}$-scale ranks are varied in the experiments as shown in Table \ref{error-sim3}. The normalized reconstruction error computed for these varying $1^{\rm st}$-scale $n$-ranks can also be seen in Table \ref{error-sim3}. As seen there, using ground truth partition provides lower-rank subtensors, and using clustering as part of the $1$-scale MS-HoSVD leads to much better approximations than HoSVD does in general.

\begin{table}[h]

\footnotesize

\centering

\caption{Mean and standard deviation for reconstruction error of low-rank approximations of $\mathcal{X}$ over 20 trials.}

\begin{tabular}{c||ccc}

\hline

     & \multicolumn{3}{|c}{Reconstruction Error} \\

\hline

0th scale rank     &$\left( 2,\;2,\;2 \right)$  &$\left( 2,\;2,\;2\right)$ & $\left( 2,\;2,\;2 \right)$
\\

1st scale rank     &$\left( 2,\;2,\;2 \right)$  &$\left( 4,\;4,\;4 \right)$ & $\left( 6,\;6,\;6 \right)$ \\ 

 \hline

Ground Truth     &   0.2502  &0.0304    &  -     \\
Partitioning &  $\pm$ 0.0263 &$\pm$ 0.0077  & - \\
 \hline
Clustering &  0.3587    & 0.1099  & 0.0254 \\
 by LSA		& $\pm$ 0.0583 & $\pm$ 0.0391   & $\pm$  0.0152 \\

 \hline
Random       & 0.6095       & 0.3588  &  0.1855    \\
Partitioning & $\pm$ 0.0398 		& $\pm$ 0.0298 &  $\pm$ 0.0251   \\
\hline
\hline
rank &$\left( 4,\;4,\;4 \right)$ & $\left( 8,\;8,\;8 \right)$ &$\left( 12,\;12,\;12 \right)$ \\ 
 \hline
truncated &  0.5457      & 0.2127  &   0.0733   \\
HoSVD      & $\pm$0.0449	 	 & $\pm$ 0.0195 & $\pm$ 0.0129     \\
 \hline
\end{tabular}
\label{error-sim3}

\end{table}

In addition, the left (LHS) and right (RHS) sides in Theorem \ref{MainThm} are computed where the first-scale projections ${\bf Q}^{(n)}_k$ each have rank $r_1=2$ and are obtained via both ground truth partitioning and clustering after the $0^{\rm th}$-scale ${\bf P}^{(n)}$ are obtained from the truncated HoSVD of $\mathcal{X}$ with $r_0 = 2$.  For comparison the $\bar{\bf P}^{(n)}$ are also computed from the truncated HoSVD of $\mathcal{X}$ with varying ranks $\bar{r}_n=\kappa r_0$ for $\kappa \in \left\lbrace 2\;,3\;,4\right\rbrace$. In Table \ref{error-sim3_1}, we report the mean value and standard deviation of both the right-hand side (RHS) and left-hand side (LHS) of the error bound in Theorem \ref{MainThm}. As seen in Table \ref{error-sim3_1}, Theorem \ref{MainThm} holds for the $1^{\rm st}$-scale approximation of $\mathcal{X}$ via MS-HoSVD since the RHS errors based on the $\bar{\bf P}^{(n)}$ projections are larger than the LHS errors no matter whether the ${\bf Q}^{(n)}_k$ are obtained via ground truth partitioning or LSA clustering.  

\begin{table}[h]
\footnotesize
\centering
\caption{Computed reconstruction error (mean and std.) corresponding to Theorem \ref{MainThm} for simulated data over 20 trials.}
\begin{tabular}{ccc||cccc}

\hline

  \multicolumn{3}{c||}{LHS}   & \multicolumn{3}{|c}{RHS} \\
\hline
0th scale rank & $\left( 2,\;2,\;2 \right)$  & & \multicolumn{3}{|c}{rank} \\

1st scale rank   &$\left( 2,\;2,\;2 \right)$& & $\left( 4,\;4,\;4 \right)$ & $\left( 6,\;6,\;6 \right)$ &$\left( 8,\;8,\;8 \right)$\\

 \hline

 Ground Truth  & 0.2618 &   &  &  &\\
 Partitioning   & $\pm$ 0.0236 &   & 2.9928 &1.1835&0.4228\\
	  \cline{1-3}
	
Clustering	 & 0.3469 &   & $\pm$0.5081&  $\pm$0.3010& $\pm$0.0978\\
by LSA  	&$\pm$ 0.0513 &     & &&\\
\hline

\end{tabular}
\label{error-sim3_1}

\end{table}

\end{document}